\documentclass[reqno,12pt]{amsart} 
\usepackage{fullpage,setspace}%
\usepackage{colonequals,tikz,amsmath,amssymb,verbatim,mathrsfs,bm}
\usepackage[UglyObsolete]{diagrams}
\usepackage[all,cmtip]{xy}

\RequirePackage[l2tabu, orthodox]{nag}
\makeatletter
\@namedef{subjclassname@2020}{%
  \textup{2020} Mathematics Subject Classification} 
\makeatother

\usepackage[unicode,
]{hyperref}

\let\oldtocsection=\tocsection
\let\oldtocsubsection=\tocsubsection
\let\oldtocsubsubsection=\tocsubsubsection
\renewcommand{\tocsection}[2]{\hspace{0em}\oldtocsection{#1}{#2}}
\renewcommand{\tocsubsection}[2]{\hspace{1em}\oldtocsubsection{#1}{#2}}
\renewcommand{\tocsubsubsection}[2]{\hspace{2em}\oldtocsubsubsection{#1}{#2}}

\newtheorem{theorem}{Theorem}[section]
\newtheorem{proposition}[theorem]{Proposition}
\newtheorem{lemma}[theorem]{Lemma}
\newtheorem{corollary}[theorem]{Corollary}
\theoremstyle{definition}
\newtheorem{definition1}[theorem]{Definition}
\newtheorem{remark1}[theorem]{Remark}

\newcommand\Idd{I}

\newcommand{\lbarbrack}{|\![}
\newcommand{\rbarbrack}{]\!|}
\newcommand{\BBE}{{\mathbb E}_2}
\newcommand{\BBEpm}{\BBE^\pm}
\newcommand{\PP}{\mathbb P}
\newcommand{\sF}{\mathsf F}
\newcommand{\sP}{\mathsf P}
\newcommand{\WF}{W_{\sF}}
\newcommand{\WP}{W_{\sP}}
\newcommand{\MF}{M}
\newcommand{\BM}{\mathbf M}

\newcommand{\oWF}{\overline{W}_{\sF}}
\newcommand{\oWP}{\overline{W}_{\sP}}
\newcommand{\oMF}{\overline{M}}

\newcommand{\CC}{\mathcal C}
\newcommand{\OO}{\mathcal O}
\newcommand{\BC}{\mathbf C} 
\newcommand{\Bc}{\mathbf c}
\newcommand{\BD}{\mathbf D}
\newcommand{\J}{\mathbf J}
\newcommand{\BK}{\mathbf K}
\newcommand{\BL}{\mathbf L}
\newcommand{\Q}{\mathbf Q}
\newcommand{\BU}{\mathbf U}
\newcommand{\Z}{\mathbf Z}
\newcommand{\oBD}{\overline{\BD}}
\newcommand{\oJ}{\overline{\J}}
\newcommand{\oBL}{\overline{\BL}}
\newcommand{\oBU}{\overline{\BU}}
\DeclareTextCommand{\textover}{PU}{\9040\076}
\newcommand{\vect}{\binom}

\definecolorset{rgb}{}{}{purpviolet,0.5,0,0.25}

\DeclareMathOperator\ncl{ncl}
\DeclareMathOperator\GLT{GL_{2}}
\DeclareMathOperator\SL{SL}
\DeclareMathOperator\SLT{\SL_{2}}
\DeclareMathOperator\SLTpm{\SLT^{\raisebox{.01in}{\!\!\ensuremath{\scriptstyle \pm}}}}
\DeclareMathOperator\SF{\mathscr{F}}
\DeclareMathOperator\SR{\mathscr{R}}
\DeclareMathOperator\length{length}
\DeclareMathOperator\olength{\overline{\length}}
\DeclareMathOperator\In{In}
\DeclareMathOperator\Rep{Rep}
\DeclareMathOperator\kk{k}
\DeclareMathOperator\CF{CF}
\DeclareMathOperator\Norm{Nm}

\DeclareMathOperator\PCF{\textup{\textsf{PCF}}}
\DeclareMathOperator\oPCF{\overline{\textup{\textsf{PCF}}}}
\DeclareMathOperator\RCF{\textup{\textsf{RCF}}}
\DeclareMathOperator\FCF{\textup{\textsf{FCF}}}
\DeclareMathOperator\oFCF{\overline{\textup{\textsf{FCF}}}}
\DeclareMathOperator\Quad{Quad}
\DeclareMathOperator\Roots{Roots}

\begin{document}

\title{A composition law and refined notions of convergence
for periodic continued fractions}


\author[B. W. Brock]{Bradley~W.~Brock}
\address{Center for Communications Research, 805 Bunn Drive, Princeton,
NJ 08540-1966, USA}
\email{brock@idaccr.org}

\author{Bruce~W.~Jordan}
\address{Department of Mathematics, Baruch College, The City University
of New York, One Bernard Baruch Way, New York, NY 10010-5526, USA}
\email{bruce.jordan@baruch.cuny.edu}

\author[L. Smithline]{Lawren Smithline}
\address{Center for Communications Research, 805 Bunn Drive, Princeton,
NJ 08540-1966, USA}
\email{lawren@idaccr.org}

\begin{abstract}
We define an equivalence relation on periodic continued fractions with
partial quotients in a ring $\OO \subseteq \BC$, a group law on these
equivalence classes, and a map from these equivalence classes to matrices
in $\GLT(\OO)$ with determinant $\pm1$.
We prove this group of equivalence classes is isomorphic to $\Z/2\Z\ast\OO$ and study
certain of its one- and two-dimensional representations.

For a periodic continued fraction with period $k$, we give a refined description
of the limits of the $k$ different $k$-decimations of its sequence of convergents.
We show that for a periodic continued fraction associated to a matrix with eigenvalues of
different magnitudes, all $k$ of these limits exist in $\PP^1(\BC)$ and a strict majority
of them are equal.
\end{abstract}

\maketitle
\tableofcontents

\section{Introduction}
Continued fractions with positive integer partial quotients
are a staple of classical number theory, with
an intimate connection to Euclid's algorithm.
Here, we consider the more general case of continued fractions with partial
quotients in a subring $\OO\ni 1$  of the complex numbers $\BC$,
as in \cite{bej}.
We define an equivalence relation on the set 
$\PCF(\OO)$ of periodic continued fractions
over $\OO$ and a group law on the resulting equivalence classes
$\oPCF(\OO)$.

The theory of periodic continued fractions, unlike continued fractions in
general, has connections to both arithmetic groups and diophantine equations.
The work \cite{bej} relates the theory of $\PCF(\OO)$ to the group
$\SLT^\pm(\OO)$, the subgroup of $\GLT(\OO)$ consisting of matrices with
determinant $\pm1$, and to diophantine problems on associated varieties.
In this paper  we further explore the relationship between $\PCF(\OO)$ and the group
$\SLT^\pm(\OO)$.
A major influence in our study is the 1966 paper \cite{Cohn} by Cohn.

Section \ref{sec:2x2} sets out preliminaries related to
actions of $2\times2$ matrices.
We consider the action of $\GLT(\OO)$ by linear fractional transformations
on the projective line $\PP^1(\BC)$.
We also name subgroups of $\GLT(\BC)$ determined
by common eigenspaces.
In Section \ref{sec:contfrac}, we introduce the 
concatentation binary operation $\star$  
on the semigroup $\FCF(\OO)$ of finite continued fractions over $\OO$.
We develop an equivalence relation on $\FCF(\OO)$ and show that
the equivalence classes $\oFCF(\OO)$ form a group  under $\star$ .
We exhibit a map $M: \FCF(\OO) \rightarrow \SLT^{\pm}(\OO)$ in
Proposition \ref{prop:mapSLT} arising
from the algorithm for computing the value of a
finite continued fraction. 
We show that $M$ is well defined on equivalence classes,
giving a homomophism $\overline{M}:\oFCF(\OO)\rightarrow \SLT^{\pm}(\OO)$.
We review
results of Cohn \cite{Cohn} that give information on
 $\ker(\overline{M})$ for certain rings $\OO$.
 Using this, we show how 
the standard amalgam presentation 
$\SLT(\Z)\cong\Z/4\Z\ast_{\Z/2\Z}\Z/6\Z$ arises from 
presenting $\SLT(\Z)$ as an explicit
quotient of $\oFCF(\Z).$

Section \ref{sec:PCF} extends the structure
on finite continued fractions to periodic continued fractions (PCFs).
In Theorem \ref{thm:PCFisom}, we show that this yields a 
natural equivalence relation
on periodic continued fractions
yielding a group of equivalence classes $\oPCF(\OO)$ isomorphic to $\oFCF(\OO)$.
Section \ref{sec:twoDreps} relates the group $\oPCF(\OO)$ to quantities defined
in \cite{bej}.
We review the fundamental matrix
$E(P)\in\SLTpm(\OO)$ associated to a periodic continued fraction $P$ and
apply the isomorphism of Theorem \ref{thm:PCFisom} to show that $E$ is well-defined
on $\oPCF(\OO)$, giving a homomophism $\overline{E}:\oPCF(\OO)\rightarrow
\SLT^{\pm}(\OO)$.
For $A\in\GLT(\OO)$ we define subgroups $\oPCF_\OO(A)$ of 
$\oPCF(\OO)$ 
by requiring $\overline{E}(\overline{P})$
to be  a linear combination of $A$ and the identity matrix $I$.
We exhibit  characters
corresponding to each eigenvalue of $A$.
These characters have values that are units in a quadratic extension of $\OO$.
In Section \ref{dirge}, we give examples where $\OO$ is a number ring.

In Section \ref{sec:convergePCF}, we reformulate the PCF-convergence
criteria of  \cite[Thm.~4.3]{bej} 
in parallel with the cases of convergence  in Section \ref{linfrac}.
We show that when $E(P)$ has eigenvalues of different magnitudes
(which implies \textsf{quasiconvergence}),
the $k$ different $k$-decimations of its sequence of convergents all converge in $\PP^1(\BC)$,
and a strict majority agree.  Unanimity is equivalent to convergence.

In Section \ref{examples}, we demonstrate the sharpness of the result of Section
\ref{sec:convergePCF}.  We give examples showing
that the proportion of agreement of
the $k$ different $k$-decimations of the sequence of convergents of a 
quasiconvergent periodic continued fraction can be arbitrarily close to $1/2$ or arbitrarily close to and different from $1$. We show 
that the equivalence relation $\sim$ on $\PCF(\OO)$ respects quasiconvergence
but does not respect convergence.

\section{Actions of \texorpdfstring{$2\times 2$}{2\texttimes2} matrices}
\label{sec:2x2}

Let $\OO\subseteq\BC$ be a ring with $1$.
A finite continued fraction with partial quotients in $\OO$
 is an iterated quotient.
It can be interpreted in terms of the action of
$\GLT(\OO)\leq \GLT(\BC)$ on the projective line $\PP^1(\BC)$ by linear
fractional transformations which we review in this section.
We also name certain subgroups of $\GLT(\BC)$ determined by common eigenspaces.

\subsection{Linear fractional transformations}
\label{linfrac}
Suppose
\begin{equation}
\label{eq:simplematrix}
M=\begin{bmatrix} m_{11} & m_{12}\\ m_{21} & m_{22}\end{bmatrix}\in\GLT(\BC)
\end{equation}
and $\beta \in \PP^1(\BC)$.
The matrix $M$ acts on $\beta$ by
linear fractional transformation
$$M\beta = \frac{m_{11}\beta + m_{12}}{m_{21}\beta + m_{22}},$$
where we interpret $1/0 = \infty\in\PP^1(\BC)$.
For $\beta\in\PP^1(\BC)$, let
\begin{equation}
\label{owl}
v(\beta)=\vect{\beta}{1}\text{ if }\beta\ne\infty\text{ and }v(\infty)=\vect10.
\end{equation}
The vector $v(\beta)$ is an eigenvector of $M$ if and only if $M\beta = \beta$.
For 
nonzero polynomial $Q=aX^2+bX+c\in\BC[X]$,
let $\Roots(Q)$ be the multiset of zeros of $Q[x]$ in $\PP^1(\BC)$, 
with the convention that if $\deg(Q)=1$, then $\infty$ is a simple root of $Q$;
and if $\deg(Q)=0$, then $0\neq Q$ has a double root at $\infty$.
When $Q = 0$, say $\Roots(Q) = \PP^1(\BC)$.
For a $2\times 2$ matrix $M$ as above, set
\begin{equation}
\label{recycling}
\Quad(M)=m_{21}X^2+ (m_{22}-m_{11})X- m_{12}.
\end{equation}
We use $\Roots(M)$ to denote $\Roots(\Quad(M))$.
When $\beta\in\Roots(M)$, $v(\beta)$ is an eigenvector of $M$ with
eigenvalue 
\begin{equation}
\label{thunder}
\lambda(\beta)\colonequals m_{21}\beta +m_{22}=m_{11}+m_{12}/\beta.
\end{equation}
In particular, $\lambda(\infty)=m_{11}$ and $\lambda(0)=m_{22}$.

\begin{proposition}
\label{prop:Qstable}
Let $A$, $B$ be $2\times 2$ matrices over $\OO$, $I$ be the $2\times 2$ identity matrix.
Let  $\kappa, \lambda \in \OO$ be not both zero.
There exists $\mu \in \OO$ such that $\kappa B = \lambda A + \mu I$
if and only if
$\kappa\Quad(B)=\lambda\Quad(A)$.
\end{proposition}

\begin{proof}
Suppose $\kappa, \lambda, \mu\in\OO$ with $(\kappa,\lambda)\neq (0,0)$ 
such that
$\kappa B = \lambda A + \mu I$.
The polynomial $\Quad(\lambda A+\mu I)$
is independent of $\mu$.
We have 
$$\kappa\Quad(B) = \Quad(\kappa B) = 
\Quad(\lambda A + \mu I) = \Quad(\lambda A) = \lambda \Quad(A).$$

Conversely, suppose $\kappa\Quad(B)=\lambda\Quad(A)$ 
for $\kappa,\lambda\in\OO$, $(\kappa,\lambda)\neq (0,0)$.
Since $\Quad$ is linear map, $\Quad(\kappa B - \lambda A) = 0$.
Thus, there is $\mu\in\OO$ such that $\kappa B-\lambda A = \mu I$.
\end{proof}

For a matrix $M$, $\Quad(M)$ and the eigenvalues of $M$
govern the convergence behavior of the sequence $M^n\beta$, $n \geq 0$,
for $\beta\in\PP^1(\BC)$.

\begin{definition1} 
\label{corned}
For any $x, y$, 
let $\delta(x,y)$ denote the diagonal matrix
$$
\delta(x,y) = 
\begin{bmatrix}x & 0 \\0 & y\end{bmatrix}.
$$
\end{definition1}

\begin{proposition}
\label{tetrapartite}
For $M\in \GLT(\BC)$, one of four mutually exclusive possibilities holds:
\begin{enumerate}
\item \label{tetrapartite1}
$\Quad(M)$ has one root $\hat\beta$ of multiplicity \textup{2}
and $\lim_{n\to\infty}M^n\beta=\hat\beta$ for all $\beta\in \PP^1(\BC)$.  
\item \label{tetrapartite2}
$\Quad(M) = 0$.  For all $\beta\in\PP^1(\BC)$, $M\beta = \beta$ and so
$lim_{n \to \infty}M^n\beta = \beta$.
\item \label{tetrapartite3}
$\Quad(M)$ has distinct roots $\beta_+, \beta_-$, and $M$ has corresponding eigenvalues
$\lambda_+=\lambda(\beta_+)$, $\lambda_-=\lambda(\beta_-)$
as in \eqref{thunder}, with $|\lambda_+| > |\lambda_-|$.
For all $\beta\ne\beta_-\in\PP^1(\BC)$, $lim_{n \to \infty}M^n\beta = \beta_+$.
For $\beta = \beta_-$, $lim_{n \to \infty}M^n\beta = \beta_-$.
\item \label{tetrapartite4}
$\Quad(M)$ has distinct roots with  $M$ having
 distinct eigenvalues of the same magnitude.  For $\beta \in \Roots(M)$,
$\lim_{n \to \infty}M^n\beta = \beta$.  Otherwise, the sequence $M^n \beta$ diverges.
\end{enumerate}
\end{proposition}

\begin{proof} Every $M$ is conjugate to its Jordan normal form, so the result follows if it is 
so for $M$ in Jordan normal form.
Let $\beta\in\PP^1(\BC)$.
The design of cases \eqref{tetrapartite1}--\eqref{tetrapartite4} makes evident 
that exactly one of the following holds, where $\lambda_1$, $\lambda_2$ are the
eigenvalues of $M$.  In cases (a) and (b) below there is one eigenvalue
$\lambda_1$ of multiplicity $2$.
\begin{enumerate}
\item
$M$ is defective, $M-\lambda_1I$ is nilpotent, and
$$\lim_{n\to\infty}\begin{bmatrix}\lambda_1 & 1\\ 0 & \lambda_1\end{bmatrix}^n\beta =\infty;$$
\item
$M = \lambda_1 I$, and
$$\lim_{n\to\infty} \delta(\lambda_1, \lambda_1)^n\beta =\beta;$$
\item
$M = \delta(\lambda_1, \lambda_2), |\lambda_1| > |\lambda_2|$, and
$$
\lim_{n\to\infty}
\delta(\lambda_1, \lambda_2)^n\beta =
\begin{cases}
\infty,\, \beta\ne0,\\[.05in]
0,\,\beta=0;\\
\end{cases}\\[0.05in]
$$
\item
$M = \delta(\lambda_1, \lambda_2), |\lambda_1| = |\lambda_2|$, $\lambda_1 \neq \lambda_2$, and
\[
\lim_{n\to\infty}
\delta(\lambda_1, \lambda_2)^n\beta =\begin{cases}
\beta,\textrm{ for } \beta \in \{0, \infty\}, \\
\text{does not exist for }\beta\notin\{0,\infty\}.
\end{cases}
\]
\end{enumerate}
\end{proof}

\subsection{Subgroups of \texorpdfstring{\except{toc}\protect{\boldmath{$\GLT(\BC)$}}\for{toc}{$\GLT(\BC)$}}
{GL\texttwoinferior(C)} determined by eigenspaces}
\label{subsec:eigenspace}

The subsets of $\GLT(\BC)$ with particular eigenvectors form subgroups of $\GLT(\BC)$.

\begin{proposition}
\label{prop:scalarmatrix}
Let $M \in \GLT(\BC)$.
Every $\beta \in \PP^1(\BC)$ satisfies $M\beta=\beta$
if and only if
$M$ is a scalar multiple of the identity matrix $I$.
\end{proposition}

\begin{proof}
Write $M=[m_{ij}]_{1\leq i,j\leq 2}$ as in \eqref{eq:simplematrix}.
If $M = m_{11}I$, then every $\beta$ satisfies $M\beta = \beta$.
Conversely, if every $\beta$ satisfies $M\beta = \beta$, then $v(0)$ and 
$v(\infty)$ are eigenvectors of $M$ and $M$ is diagonal.
Since $v(1)$ is also an eigenvector, 
$M$ must be a scalar multiple of the identity.
\end{proof}

\begin{proposition}
\label{prop:singlevector}
Let $0\neq L(X)\in\BC[X]$, $\deg(L)\leq 1$,
with root $\beta\in \PP^1(\BC)$. Set
$T(\beta) = \{M \in \GLT(\BC) \mid M\beta = \beta \}$.
\begin{enumerate}
\item
\label{day1}
 The set  $T(\beta)$ is a subgroup of $\GLT(\BC)$ 
conjugate to the Borel subgroup of $\GLT(\BC)$
consisting of upper triangular matrices.
\item
\label{day2}
For every $M\in T(\beta)$, $L(X)$ divides $\Quad(M)$.
If $\deg L(X)=0$, $L(X)|\Quad(M)$ means $\deg \Quad(M)\leq 1$.
\item
\label{day3}
The group $T(\beta)$ has a multiplicative character $\lambda_\beta$
mapping $M$ to its $v(\beta)$-eigenvalue.
\end{enumerate}
\end{proposition}

\begin{proof}
\eqref{day1}: The stabilizer $T(\infty)$ of $\infty\in\PP^1(\BC)$
consists of the upper triangular matrices.  Since $\GLT(\BC)$
acts transitively on $\PP^1(\BC)$, $T(\beta)$ is a subgroup of 
$\GLT(\BC)$ conjugate to $T(\infty)$ for every $\beta\in\PP^1(\BC)$.\\
\eqref{day2}: Suppose $M\in T(\beta)$.  Then $\beta\in\Roots(M)$
and $L(X)$ divides $\Quad(M)$.\\
\eqref{day3}:  Suppose $M, M'\in T(\beta)$.  Then $MM'\in T(\beta)$ and
\[
\lambda_\beta(MM')v(\beta)=MM'v(\beta)=M\lambda_\beta(M')v(\beta)=
\lambda_\beta(M)\lambda_\beta(M')v(\beta),
\]
so $\lambda_\beta$ is a multiplicative character on $T(\beta)$.
\end{proof}

\begin{proposition}
\label{prop:Qgroup}
Let $0\neq Q\in\BC[X]$ with $\deg(Q)\leq 2$. 
The set $$G(Q) = \{M \in \GLT(\BC) \mid \exists \lambda\in\BC\text{ such that } \Quad(M) = \lambda Q \}$$ is a group.
When $Q$ is not a square, $$G(Q) = T(\beta) \cap T(\beta^*)
\text{ with }\{\beta, \beta^* \}= \Roots(Q), \,\,\beta\neq\beta^\ast,$$
and $G(Q)$ is conjugate to the group of diagonal matrices in $\GLT(\BC)$.
When $Q$ is a square in $\BC[X]$, $G(Q)$ is conjugate to the 
the subgroup of upper triangular matrices in $\GLT(\BC)$ with equal diagonal entries.
\end{proposition}

\begin{proof}
When $Q$ is not a square, $\Roots(Q) = \{\beta, \beta^*\}$, $\beta \neq \beta^*$.
In this case, 
$M\in G(Q)$ is equivalent to $M$ having eigenvectors 
$v(\beta)$ and $v(\beta^*)$.
Hence, $G(Q) = T(\beta) \cap T(\beta^*)$. A change of basis mapping $\beta$ to $\infty$
and $\beta^*$ to $0$ conjugates $G(Q)$ to $G(x)$, the group of diagonal matrices.

When $Q$ is a square, $\Roots(Q)$ has one element $\beta$ of multiplicity $2$..
Every $M\in G(Q)$ is either a scalar multiple of $I$ or defective.
A change of basis mapping $\beta$ to $\infty$
conjugates $G(Q)$ to $G(1)$.  A matrix in $G(1)$ is either a scalar multiple of $I$ or
upper triangular with equal diagonal entries and not diagonal.
\end{proof}

\section{Finite continued fractions}
\label{sec:contfrac}
Let $\OO\subseteq \BC$ be a ring containing 1.
We define the set $\FCF(\OO)$ of finite continued fractions
with partial quotients in $\OO$,
and an operation
$\star$ making $\FCF(\OO)$ a semigroup.
We give an equivalence relation on $\FCF(\OO)$ and show that the equivalence classes
$\oFCF(\OO)$ form a group.

Let $c_i\in\BC$, $1\leq i\leq n$.

\begin{definition1}
\label{sun}
A \textsf{finite continued fraction} (FCF) $F=
[c_1,c_2, \ldots, c_n]$ is the formal expression
\begin{equation}
\label{contfrac}
 c_1 + \cfrac{1}{c_2+\cfrac{1}{c_3+\cfrac{1}{c_4+\cfrac{1}
{\raisebox{-1.2em}{\ensuremath{\ddots\quad}}\raisebox{-2.4em}
{\ensuremath{c_{n-1}+\cfrac{1}{c_n}}}}}}}.
\end{equation}
We say that the FCF $F=[c_1,\ldots, c_n]$ has \textsf{length} 
$\length(F)=n$.
\end{definition1}
The elements $c_i$ are called the \textsf{partial quotients}.  
For $k \leq n$, the \textsf{convergent}
$\CC_k(F)$ of $F$ is the evaluation of the finite continued fraction $[c_1,c_2,\ldots, c_k]$, 
where we interpret $1/0$ as $\infty \in \PP^1(\BC)$.
We distinguish between the
formal object $F = [c_1, c_2, \ldots c_k]$ and its value 
$\hat{F}\colonequals\CC_k(F) \in \PP^1( \BC)$.
There is a simple rule to iteratively compute $\CC_k(F)$ as a ratio $p_k/q_k$.  Let
\begin{equation*}
D(x)=\begin{bmatrix} x & 1 \\
1 & 0\end{bmatrix}.
\end{equation*}
Let $\Idd$ be the $2\times 2$ identity matrix and $F=[c_1, \ldots, c_n]$
be a finite continued fraction as in \eqref{contfrac}.
Define $p_{0}, q_{0}, p_{-1}, q_{-1}$ by the matrix equation
$$
\begin{bmatrix} p_0 & p_{-1} \\
q_0 & q_{-1}\end{bmatrix}
= \Idd .$$
Recursively define $p_k=p_k(F)$, $q_k=q_k(F)$ for  $n\geq k \geq 1$ by 
$$\begin{bmatrix} p_k & p_{k-1} \\
q_k & q_{k-1}\end{bmatrix}
=
\begin{bmatrix} p_{k-1} & p_{k-2} \\
q_{k-1} & q_{k-2}\end{bmatrix}
D(c_k)=D(c_1)\cdots D(c_k).$$
The numerators $p_k$ and denominators $q_k$ 
with $\CC_k(F)=p_k(F)/q_k(F)$
can be written in terms of continuant polynomials (see, e.g.,
\cite[p.~385]{bej}).

An $\OO$-FCF is a FCF 
$[c_1,\ldots, c_n]$ having all partial quotients $c_i\in\mathcal{O}$.

\subsection{An equivalence relation on finite continued fractions}
\label{harbor}

\begin{definition1}
The semigroup $\FCF(\OO)$ is the set 
$\{F\mid F\text{ is an $\OO$-FCF}\}$ with the concatenation
operation $\star$:
\[
[c_1,\ldots , c_n]\star[c'_1,\ldots , c'_{n'}]=
[c_1,\ldots , c_n][c'_1,\ldots , c'_{n'}]\colonequals
[c_1,\ldots, c_n, c_1',
\ldots, c_{n'}].
\]
\end{definition1}
Equivalently, $\FCF(\OO)$ is the semigroup of the 
set of words in formal symbols $\BD(x)$, for $x \in \OO$, with the
two descriptions identified by
\[
[x_1,\ldots, x_n]\leftrightarrow \BD(x_1)\cdots \BD(x_n).
\]

\begin{definition1}
\label{muster}
Let $\oFCF(\OO)$ be the set $\FCF(\OO)$
modulo the equivalence relation $\sim$ generated by the relations
\begin{equation}
\label{eq:Dsinglerelation}
\BD(x)\BD(0)\BD(y) = \BD(x+y),\  x, y \in \OO.
\end{equation}
For $F=[c_1,\ldots, c_n]\in \FCF(\OO)$ denote by $\overline{F}=
\lbarbrack c_1,\ldots, c_n\rbarbrack\in\oFCF(\OO)$ the equivalence
class containing $F$. 
For a word $\mathbf{w}$ in the $\mathbf{D}(x)$'s denote by 
$\overline{\mathbf{w}}$ the equivalence class containing $\mathbf{w}$.
\end{definition1}

\begin{proposition}
\label{mush}
\begin{enumerate}
\item
\label{mush1}
If $F_1,G_1, F_2, G_2 \in\FCF(\OO)$ with
$\overline{F}_1=\overline{F}_2$ and $\overline{G}_1=\overline{G}_2$,
then 
$$\overline{F_1\star G_1}=\overline{F_2\star G_2}\in\oFCF(\OO).$$
\item
\label{mush2}
The binary operation $\star$ on $\FCF(\OO)$ induces a well-defined
binary operation $\star$ on $\oFCF(\OO)$, making $\oFCF(\OO)$ a semigroup.
\item
\label{mush3}
The semigroup $\oFCF(\OO)$ is a group with identity $\lbarbrack 0,0\rbarbrack
\leftrightarrow \overline{\BD}(0)^2$.
\item
\label{mush4}
The element $\overline{\BD}(x)\in\oFCF(\OO)$ has inverse 
$\overline{\BD}(x)^{-1}=\overline{\BD}(0)\overline{\BD}(-x)\overline{\BD}(0)$.

\end{enumerate}
\end{proposition}

\begin{proof}
\eqref{mush1} is immediate, and \eqref{mush2} is implied by \eqref{mush1}.\\
\eqref{mush3}:
Substituting $y = 0$ into relation \eqref{eq:Dsinglerelation} shows that 
the class of $\BD(0)^2$
is the identity in $\oFCF(\OO)$.  The equality
$$\BD(x)\BD(0)\BD(-x)\BD(0) = \BD(0)^2, \ x \in \OO,$$
exhibits the inverse $\overline{\BD}(x)^{-1}=
\overline{\BD}(0)\overline{\BD}(-x)\overline{\BD}(0)$.
\end{proof}


\begin{definition1}
\label{pork}
Define $\J,\BU(x),\BL(x)\in\FCF(\OO)$ for $x\in \OO$ by
$\J\colonequals \BD(0)$, $\BU(x)\colonequals \BD(x)\BD(0)$,
and $\BL(x)\colonequals \BD(0)\BD(x)$.
We consider the group $\Z/2\Z\ast\OO$,
where $\OO$ is viewed as an additive group and $j\in\Z/2\Z$ is a generator.
Let $\WF:\FCF(\OO)\rightarrow \Z/2\Z \ast \OO$ be the surjective
map of semigroups
\[
\WF:\BD(x)\mapsto xj\in \Z/2\Z\ast\OO.
\]
  In particular, note
that 
$\WF(\J)=j\in\Z/2\Z$, $\WF(\BU(x))=x\in\OO$,
and $\WF(\BL(x)))=jxj$.
\end{definition1}

\begin{proposition}
\label{prop:firstfreeprod}
The morphism of semigroups $\WF:\FCF(\OO)\rightarrow\Z/2\Z\ast\OO$
of Definition \textup{\ref{pork}} induces an isomorphism of 
groups $\oWF:\oFCF(\OO)\stackrel{\simeq}{\longrightarrow}\Z/2\Z\ast\OO$.
\end{proposition}

\begin{proof}
Verify that $\WF(\BD(x)\BD(0)\BD(y))=\WF(\BD(x+y))$ for $x,y\in\OO$.
Hence $\WF$ induces a well-defined map on equivalence
classes $\oFCF(\OO)=\FCF(\OO)/\sim$ by the Definition \ref{muster}
of  $\sim$. To see that $\oWF$ is an isomorphism, verify that
its inverse is given by
\[
\oWF^{\,-1}:\Z/2\Z\ast\OO\longrightarrow \oFCF(\OO)\quad\text{with}\quad
\oWF^{\,-1}(j)=\oJ\text{ and } \oWF^{\,-1}(x)=\oBU(x)\text{ for }x\in\OO.\qedhere
\]
\end{proof}

\begin{remark1}
\label{feast}
For $F=[c_1,\ldots, c_n]\in\FCF(\OO)$, set 
$F^\ast\colonequals [0,-c_n,\ldots, -c_1,0]$. 
Observe that
 $\overline{F^\ast}=
\overline{F}^{\,-1}\in
\oFCF(\OO)$\textup{:}
$$
\lbarbrack c_1,\ldots, c_n\rbarbrack\star \lbarbrack 0, -c_n, \ldots,
-c_1,0\rbarbrack 
=\lbarbrack c_1,\ldots, c_n\rbarbrack \lbarbrack 0, -c_n, \ldots,
-c_1,0\rbarbrack 
=\lbarbrack 0,0\rbarbrack . $$
This formula for $\overline{F}^{\,-1}$ also follows from
Proposition \ref{mush}\eqref{mush4}.
\end{remark1}

\begin{definition1}
\label{stunt}
The finite continued fraction $F=[c_1,\ldots ,c_n]$ is \textsf{reduced} if
it has no interior zeros:
$c_i\ne 0$ for $2\leq i\leq n-1$.
\end{definition1}

\begin{proposition}
\label{prop:reducedFCF}
For $F \in \FCF(\OO)$, there is a unique reduced $F_\mathrm{red}\in \FCF(\OO)$
such that $F_\mathrm{red} \sim F.$
\end{proposition}

\begin{proof}
Given $F \in \FCF(\OO)$, iteratively apply relation \eqref{eq:Dsinglerelation}
to remove interior zeros to produce a reduced $F_\textrm{red} \sim F$.

On the other hand, suppose $F_\textrm{red}, F'_\textrm{red}\in\FCF(\OO)$
are both reduced and $\overline{F}_\textrm{red}=\overline{F'}_\textrm{red}$.
Then $\oWF(\overline{F}_\textrm{red})=
\oWF(\overline{F'}_\textrm{red})\in\Z/2\Z\ast\OO$.
Since they are reduced, both $\overline{F}_\textrm{red}$ and 
$\overline{F'}_\textrm{red}$ map under $\oWF$ 
to words in $\oFCF(\OO) \cong \Z / 2\Z \ast \OO$,
each literally expressed without spurious insertions of the identity, and they are equal.
Since $\Z / 2\Z \ast \OO$ is a free product, these words are the same and
$F'_\textrm{red} = F_\textrm{red}$.
\end{proof}

\begin{definition1}
\label{def:FCFnormalform}
For any 
 $\overline{F}\in\oFCF(\OO)$, the \textsf{normal form} of $\overline{F}$ is
the unique reduced representative $F_\textrm{red} \in \FCF(\OO)$
of the equivalence class $\overline{F}$.
\end{definition1}

\begin{remark1}
\label{gifted}
Two elements $\overline{F}, \overline{F}'\in\oFCF(\OO)$ are equal if and only if they have
the same normal forms in $\FCF(\OO)$: 
$$\overline{F}=\overline{F}'\Longleftrightarrow F_\textrm{red}=F'_\textrm{red}\in
\FCF(\OO).$$

In practice one computes $F_{\mathrm{red}}$ for $F\in\FCF(\OO)$ by
the algorithm of repeatedly applying relation \eqref{eq:Dsinglerelation}
until all interior zeros are eliminated.
For example, consider 
$$F= [0,-2,0,2,0,3,0,5]\in\FCF(\Z).$$
We have
$$
F=[0,-2,0,2,0,3,0,5]\sim [0,0,0,3,0,5]\sim[0,3,0,5]\sim[0,8],
$$
and hence $F_{\mathrm{red}}=[0,8]\in\FCF(\Z)$ with
$\overline{F}=\lbarbrack 0,-2,0,2,0,3,0,5\rbarbrack =
\lbarbrack 0,8\rbarbrack =\overline{F}_\textrm{red}\in\oFCF(\Z)$.
\end{remark1}

\subsection{Mapping finite continued fractions to matrices}
\label{subsec:FCF2matrix}
Let 
\begin{equation*}
\SLTpm(\OO) = \{g\in \GLT(\OO)\mid \det(g)=\pm 1\}.
\end{equation*}
The kernel of the determinant map, $\SLT(\OO)$, is normal of index 2 in $\SLTpm(\OO)$.

\begin{proposition}
\label{prop:mapSLT}
The map 
$\MF:\FCF(\OO)\rightarrow \SLTpm(\OO)$ given by
\begin{equation}
\label{eq:FCFtoSLT}
\MF([c_1,\ldots ,c_n])=D(c_1)\cdots D(c_n)\quad\text{so that}\quad
\MF(\BD(x)) = D(x)
\end{equation}
induces a well-defined homomorphism of groups
$\oMF:\oFCF(\OO)\rightarrow \SLTpm(\OO)$.
We therefore have a homomorphism of groups
$\BM\colonequals \overline{M}\circ\oWF^{\,-1}:\Z/2\Z\ast\OO\rightarrow \SLTpm(\OO)$.
\end{proposition}

\begin{proof}
We have specified the map $\MF$ on $\FCF(\OO)$.  We must show that $\MF$ respects
equivalence under  \eqref{eq:Dsinglerelation}.
By direct computation,
\[\MF(\BD(x)\BD(0)\BD(y)) = D(x)D(0)D(y) = D(x+y) = \MF(\BD(x+y)).\qedhere\]
\end{proof}

\begin{proposition}
\label{fest}
Suppose $F_1, F_2\in \FCF(\OO)$ and $F_1 \sim F_2$.
\begin{enumerate}
\item
\label{fest1}
The values
of $F_1$ and $F_2$ are equal: $\hat{F}_1=\hat{F}_2$.
\item
\label{fest2}
We have $\length(F_1)\equiv \length(F_2) \bmod 2$, so
there is a well-defined function  
\[
\olength:\oFCF(\OO)\rightarrow 
\Z/2\Z\quad\text{given by}\quad \olength(\overline{F})=\length(F)\pmod 2. 
\]
\end{enumerate}
\end{proposition}

\begin{proof}
\eqref{fest1}: Since $F_1\sim F_2$, 
$\MF(F_1)=\MF(F_2)\in \SLTpm(\OO)$ by Proposition \ref{prop:mapSLT}
We have that $\hat{F}_1=\hat{F}_2$, because $\hat{F}_1
=\MF(F_1)_{11}/\MF(F_1)_{21}$
and likewise for $\hat{F}_2$ by \eqref{eq:FCFtoSLT}.\\
\eqref{fest2}: Reformation of a word by relation \eqref{eq:Dsinglerelation}
maintains the parity of the word length. 
\end{proof}

\begin{proposition}
\label{alin}
Let 
$\chi:\oFCF(\OO)\rightarrow \langle\pm 1\rangle$ be 
$\det \circ\, \overline{M}$.
Let $\tilde{\chi}:\Z/2\Z\ast\OO\rightarrow \langle\pm 1\rangle$
be the group homomorpism to $\langle\pm 1\rangle$ that maps $j$ to $-1$ 
and is trivial on the second factor.
For $w\in \FCF(\OO)$, we have $\chi(\overline w) = (-1)^{\olength(\overline{w})}$
and $\tilde{\chi}\circ\oWF=\chi$.
\end{proposition}

\begin{proof}
The maps $\chi$ and $\tilde{\chi}\circ\oWF$ are
 determined by their action on the generators 
$\overline{\BD}(x)$ of $\oFCF(\OO)$, namely
$\chi(\oBD(x)) = -1$ and $\tilde{\chi}\circ\oWF(\oBD(x))=\tilde{\chi}(xj)=-1$.
Hence $\chi=\tilde{\chi}\circ\WF$ and 
for $w \in \FCF(\OO)$,
 $\chi(\overline{w}) = (-1)^{\length(w)}= (-1)^{\olength(\overline{w})}  $.
\end{proof}

\begin{definition1}
\label{diana}
Set $\oFCF^{\,+}(\OO)=\ker(\olength)\triangleleft\oFCF(\OO)$.
\end{definition1}

\begin{proposition}
\label{prop:FCFplus}
Let $\tilde{\chi}:\Z/2\Z\ast\OO\rightarrow \langle \pm 1 \rangle$
be as in Proposition \textup{\ref{alin}}.  There is an isomorphism
$\psi:\OO\ast\OO\rightarrow \ker{\tilde{\chi}}\subseteq \Z/2\Z\ast\OO$
given by sending $x\in\OO$ in the first factor to $x\in\Z/2\Z\ast \OO$
and $y\in\OO$ in the second factor to $jyj\in\Z/2\Z\ast\OO$.
Hence $\oWF$ maps $\oFCF^{\,+}(\OO)$ isomorphically onto $\OO\ast\OO\simeq
\ker(\tilde{\chi})$.
\end{proposition}

\begin{proof}
The definition of the free product $\Z/2\Z\ast\OO$ implies that $\psi$ is
injective.
A word in $\Z/2\Z\ast\OO$ with an even number of $j$'s is of the form
$x_1jy_1jx_2jy_2j\cdots x_njy_nj$ with possibly $x_1$, $y_n$, or both 0,
which is in the image of $\psi$.
\end{proof}

\begin{definition1}
\label{rent}
Set  
\[
W_{\sF^+}\colonequals \psi^{-1} \circ W_{\sF}|_{\FCF^+(\OO)}:
\FCF^+(\OO)\longrightarrow \OO\ast\OO.
\]
\end{definition1}

\begin{proposition}
\label{goat}
\begin{enumerate}
\item
\label{goat1}
The map $W_{\sF^+}$ of Definition \textup{\ref{rent}}
induces an isomorphism 
\[
\overline{W}_{\sF^+}:\oFCF^{\, +}(\OO)\stackrel{\simeq}{\longrightarrow} 
\OO\ast\OO.
\]
\item
\label{goat2}
Let $\J, \BU(x),  \BL(x)\in \FCF(\OO)$ for $x\in\OO$  be as in
Definition \textup{\ref{pork}}.
We have $\overline{W}_{\sF^+}(\oBU(x))=x$ in the first factor of $\OO\ast\OO$
and $\overline{W}_{\sF^+}(\oBL(y))=y$ in the second factor.
In particular $\{\oBU(x),\oBL(y)\mid x,y\in\OO\}$ generates
$\oFCF^{\, +}(\OO)$.
\end{enumerate}
\end{proposition}

\begin{proof}
The fact that the equivalence relation does not change the parity of the number
of $\J$'s implies \eqref{goat1}.
\eqref{goat2} follows from Proposition \ref{prop:FCFplus} after unwinding
the definitions.
\end{proof}

\begin{remark1}
\label{butter}
The map $\oMF$ of Proposition \ref{prop:mapSLT}
on $\oFCF(\OO)$ takes $\oJ$ to
$J \colonequals D(0) = 
\left[\begin{smallmatrix}0 & 1\\1&0\end{smallmatrix}\right].$
For $x\in\OO$,
\begin{equation*}
\oMF(\overline{\mathbf{D}}(x))=D(x)\colonequals \begin{bmatrix}
x & 1 \\ 1 & 0\end{bmatrix},\,\,
\oMF(\oBU(x)) = U(x)\colonequals \begin{bmatrix} 1 & x\\ 0 & 1 
\end{bmatrix},\,\, \oMF(\oBL(x))= L(x) \colonequals
\begin{bmatrix} 1 & 0 \\
x & 1\end{bmatrix}.
\end{equation*}
\end{remark1}

\begin{definition1}
\label{undo}
\begin{enumerate}
\item
The \textsf{elementary matrices} in $\SLT(\OO)$ are 
$\{L(c), U(c)\mid  c\in \OO\}$.
\item
The \textsf{elementary matrices} in $\SLTpm(\OO)$ are 
$\{L(c), U(c), J\mid c\in\OO\}$.
\item
The group $\BBE(\OO)$ is the subgroup of $\SLT(\OO)$ generated by elementary matrices in $\SLT(\OO)$.
\item
The group $\BBEpm(\OO)$ is the subgroup of $\SLT^\pm(\OO)$ generated by elementary matrices
in $\SLT^\pm(\OO)$.
\item
\label{undo1}
Set $\overline{M}^+=\overline{M}|_{\oFCF^{\,+}(\OO)}:\oFCF^{\,+}(\OO)\longrightarrow
\SLT(\OO)$.
\end{enumerate}
\end{definition1}

\begin{proposition}
\label{planet}
The image of $\oMF:\oFCF(\OO)\rightarrow \SLTpm(\OO)$ 
is $\BBEpm(\OO)$
and $\oMF(\oFCF^{\, +}(\OO))=\BBE(\OO)$.
\end{proposition}

\begin{proof}
By Remark \ref{butter},
$\oMF$ maps generators of 
$\oFCF(\OO)$ to generators of $\BBEpm(\OO)$
and generators of $\oFCF^{\, +}(\OO)$ to generators of $\BBE(\OO)$.
\end{proof}

\begin{corollary}
We have
$\BBEpm(\OO)=\langle D(x) : x\in\OO \rangle$ and
$\BBEpm(\OO)\cap\SLT(\OO)=\BBE(\OO).$
\end{corollary}

\begin{proof}
Since $\oFCF(\OO)$ is generated by $\BD(x)$, the equality
$\BBEpm(\OO)=\langle D(x) : x\in\OO \rangle$ is an
immeditate consequence of Proposition \ref{planet}.
The kernel of the determinant map applied to $\BBEpm(\OO)$ is $\BBEpm(\OO)\cap\SLT(\OO)$.
The group $\oFCF^{\, +}(\OO)$ is the kernel of determinant composed with 
$\oMF$ by Proposition \ref{alin} and Definition \ref{diana}. 
Thus, the equality
$\BBEpm(\OO)\cap\SLT(\OO)=\BBE(\OO)$ follows.
\end{proof}

\begin{remark1}
Note that $\BBEpm(\OO)=\SLTpm(\OO)$ if and only if 
$\BBE(\OO)=\SLT(\OO)$.
For a number field $K$ with integers $\OO_K$, it is known
that $\BBE(\OO_K)=\SLT(\OO_K)$ unless $K=\mathbf{Q}(\sqrt{-D})$
where $D\ne 1,2,3,7,11$ and is squarefree.
Vaser\v{s}te\u{\i}n \cite{Vas} proved that $\BBE(\OO_K)=\SLT(\OO_K)$
if $K$ is {\em not} imaginary quadratic (see also \cite{Lie}), and Cohn
\cite[Thm.~6.1]{Cohn} had previously settled the imaginary quadratic case.

Nica \cite{Nica} provides a short constructive proof that for $\OO_K$
the ring of integers of an imaginary quadratic field $K$,
either $\BBE(\OO_K)=\SLT(\OO_K)$ or $\BBE(\OO_K)$
is an infinite index non-normal subgroup of $\SLT(\OO_K)$.
\end{remark1}

\subsection{The kernel of \texorpdfstring{\except{toc}{\boldmath
{$\overline{M}$}}\for{toc}{$\overline{M}$}}{M} and presentations of
 \texorpdfstring{\except{toc}{\boldmath{$\BBE(\OO)$}}\for{toc}
{$\BBE(\OO)$}}{E\texttwoinferior(O)}}
\label{rotten}

The map $\oMF:\oFCF(\OO)\rightarrow \SLTpm(\OO)$
of Proposition \ref{prop:mapSLT}
has image $\BBEpm(\OO)$ by 
Proposition \ref{planet}. It is a difficult
problem to find the kernel of $\oMF$.  We begin with the 
observation that $\oMF$ is not injective.

\begin{remark1}
\label{stress}
Let 
$$x, y, z \in \OO \quad\text{with}\quad w = -x - z - xyz.$$
Suppose $x, y, z \neq 0$, and there exists $a, b \in \OO$ such that
$xy = aw$, $yz = bw$.
Note $\OO$ is an integral domain since $\OO\subset\BC$, so $x,y\ne0$
implies $w\ne0$.
We also have $y+a+b+awb=0$ because
$w(y+a+b+awb)=wy+xy+yz+xy^2z=y(w+x+z+xyz)=0$.
This is sufficient to compute that

\begin{align*}
\overline{M}(\lbarbrack x,y,z, a, w, b\rbarbrack) &=
\left( D(x)D(y)D(z)\right)\left( D(a)D(w)D(b)\right) \\
&= \begin{bmatrix}-w & aw+1\\bw+1 & y\end{bmatrix}
\begin{bmatrix}-y & aw+1\\bw+1 & w\end{bmatrix}=1\in\SLT(\OO).
\end{align*}
Hence, for example we have the following nontrivial elements in the kernel of
$\oMF$:
\begin{align}
\nonumber\lbarbrack x, -x^{-1}, x&,x^{-1}, -x, x^{-1}\rbarbrack, \\
\label{sixterm4}\lbarbrack x, -4x^{-1}, x&, -2x^{-1}, 2x, -2x^{-1}\rbarbrack,\\ 
\label{sixterm3}\lbarbrack x, -3x^{-1}, x&, -3x^{-1}, x, -3x^{-1}\rbarbrack,\\
\label{sixterma}\lbarbrack x, \alpha, -\alpha^{-1}&,x\alpha^2, 
\alpha^{-1},-\alpha\rbarbrack,
\end{align}
provided $1/x,2/x,3/x,1/\alpha\in\OO$, respectively.

A ring $\OO \subset \BC$ has a norm $|\cdot|$ given by the usual absolute value.

\begin{definition1}
A ring $\OO \subset \BC$  is \textsf{discretely normed} \cite[p.~16]{Cohn}
when every unit $x \in \OO^\times$ has norm 1, and
every nonzero, nonunit $x \in \OO$ has norm at least $2$.
\end{definition1}

An imaginary quadratic number ring $\OO$ is discretely normed unless it has
discriminant
$$-D \in \{-3,-4,-7,-8,-11,-12\}.$$
The exceptions arise from elements $x$ such that $|x|=\sqrt{2}$ when
$-D=-4,-7,-8$ or $|x|=\sqrt{3}$ when $-D=-3,-8,-11,-12$, namely
\[
x=1+i,\, \frac{1+\sqrt{-7}}2,\, \sqrt{-2}, \,\quad\text{and}\quad
x=\frac{3+\sqrt{-3}}2,\, 1+\sqrt{-2},\, \frac{1+\sqrt{-11}}2,\, \sqrt{-3}
\]
up to signs.
These all give nontrivial kernel elements when substituted into
expressions \eqref{sixterm4} and \eqref{sixterm3}, respectively,
such that all partial quotients are nonzero nonunits.

If $\OO$ is any discretely normed ring
and $\lbarbrack c_1, c_2, \ldots, c_n\rbarbrack\in\ker \oMF$, 
then some $c_i$ is a unit or zero as shown by 
a modification of 
\cite[Lemma 5.1]{Cohn}.
\end{remark1}

In parallel with \cite[p. 27]{Cohn}, we define the following subgroup of
$\oFCF(\OO).$

\begin{definition1}
\label{def:ko}
\begin{enumerate}
\item
\label{def:bc}
For $a \in \OO^\times$, let $\Bc(a) = \lbarbrack a, -a^{-1}, a
\rbarbrack \in \oFCF(\OO).$
\item
\label{def:ko2}
Let $\BK(\OO)$ be the normal closure in $\oFCF(\OO)$ of the group generated by
\begin{equation}
\label{eq:c1square}
\Bc(1)^2,
\end{equation}
\begin{equation}
\label{eq:cproduct}
\Bc(a)\Bc(b)\Bc(a^{-1}b^{-1})\Bc(1), \ a,b\in \OO^\times,
\end{equation}
\begin{equation}
\label{eq:ccommute}
\oBD(x)\Bc(a)\oBD(0)\oBD(a^2x)\oBD(0)\Bc(-a), \ a\in\OO^\times, x\in\OO.
\end{equation}
\end{enumerate}
\end{definition1}
Calculating that
$\oMF(\Bc(a)) = \delta(a, -a^{-1})$ in the notation of Definition \ref{corned},
it is routinely verified that \eqref{eq:c1square} -- \eqref{eq:ccommute}
map to $I\in\SLTpm(\OO)$ under $\oMF$  and hence 
$\BK(\OO)$ is a normal subgroup of $\ker \oMF$.
Relations \eqref{eq:c1square} and \eqref{eq:cproduct}
give $\langle \Bc(a) : a \in \OO^\times \rangle \,\,\text{mod}\, \BK(\OO)\, \subset \oFCF(\OO)/\BK(\OO)$
the same commutative group structure as
$\langle \oMF(\Bc(a)) \rangle \subset \SLT^\pm(\OO)$.
In particular,
\[
\Bc(1)\Bc(-1)\,\,\text{mod}\,\BK(\OO)=\Bc(-1)\Bc(1)\,\,\text{mod}\, \BK(\OO)
\]
has order $2$ and maps by $\oMF$ to $-I$. 
Relation \eqref{eq:ccommute} implies
\[
\Bc(-a)\oBD(x)\Bc(a)=\oBD(-a^2x)\bmod\BK(\OO), 
\]
from which we
see that $\Bc(1)\Bc(-1)\bmod\BK(\OO)$ is central in
$\oFCF(\OO)/\BK(\OO)$.

\begin{definition1}
\label{soak}
Borrowing the terminology of \cite[p.~8]{Cohn}, we say a ring $\OO$ is
\textsf{universal} for $\BBE^\pm$ when $\oFCF(\OO) / \BK(\OO) \cong \BBE^\pm(\OO)$,
in other words, if $\ker \overline{M}=\BK(\OO)$.
\end{definition1}
\begin{remark1} 
\label{sunset}
Note that taking $a=-1, b=1$ in \eqref{eq:cproduct}
gives $\Bc(-1)\Bc(1)\Bc(-1)\Bc(1)\in\BK(\OO)$ for any $\OO$,
implying the same for its conjugate $\Bc(1)\Bc(-1)\Bc(1)\Bc(-1)$.
Set  
\begin{equation}
\label{east}
\mathbf{K'}(\OO)\colonequals \ncl  \langle \Bc(1)\Bc(-1)\Bc(1)\Bc(-1), 
\lbarbrack x, 1,-1,x,1,-1\rbarbrack, \lbarbrack x, -1,1,x,-1,1\rbarbrack, x\in\OO\rangle ,
\end{equation}
the normal closure being taken in $\oFCF(\OO)$.
Note that $\Bc(1)^2,\, \Bc(-1)^2\in
\BK'(\OO)$
by taking $x=1$, $x=-1$ in \eqref{east}.
  From \eqref{eq:cproduct}, \eqref{eq:ccommute}
we see that $\BK'(\OO)\lhd\, \BK(\OO)$.
When $\OO^\times =\langle \pm 1\rangle$,
we have $\BK'(\OO)=\BK(\OO)$.
\end{remark1}

\begin{proposition}
\label{forgot}
Let
\begin{equation}
\label{bitten}
\BK''(\Z)=\ncl\langle \Bc(1)^2, \Bc(-1)^2,\Bc(1)\Bc(-1)\Bc(1)\Bc(-1)\rangle,
\end{equation}
the normal closure being taken in $\oFCF(\Z)$.
Then $\BK(\Z)=\BK'(\Z)=\BK''(\Z)$ and
$\BK(\Z)$ is a normal subgroup of $\oFCF^{\,+}(\Z)$, equal to 
the normal closure of 
$\langle \Bc(1)^2, \Bc(-1)^2,\Bc(1)\Bc(-1)\Bc(1)\Bc(-1)\rangle$
in $\oFCF^{\, +}(\Z)$.
\end{proposition}
\begin{proof}
To show that $\BK''(\Z)=\BK'(\Z)=\BK(\Z)$, it suffices
to show that 
\begin{equation}
\label{dinner}
\langle  
\lbarbrack x, 1,-1,x,1,-1\rbarbrack, \lbarbrack x, -1,1,x,-1,1\rbarbrack, x\in\OO
 \rangle\subseteq
\BK''(\Z) 
\end{equation}
in light of Remark \ref{sunset}.
Verify that
\begin{equation*}
\label{sandy}
\Bc(1)^2\lbarbrack x\rbarbrack^{-1}\lbarbrack x-1,1,-1,x-1,1,-1\rbarbrack\lbarbrack x
\rbarbrack = \lbarbrack x\rbarbrack^{-1}\lbarbrack
x,1,-1,x,1,-1\rbarbrack\lbarbrack x\rbarbrack
\end{equation*}
for $x\in\OO$, and hence 
\begin{equation*}
\lbarbrack x-1,1,-1,x-1,1,-1\rbarbrack\in \BK''(\Z)\Longleftrightarrow
\lbarbrack
x,1,-1,x,1,-1\rbarbrack\in\BK''(\Z).
\end{equation*}
Since $\Bc(-1)^2=\lbarbrack -1,1,-1, -1,1,-1\rbarbrack\in\BK''(\Z)$, it follows by induction that
\[
\lbarbrack x,1,-1, x,1,-1\rbarbrack \in\BK''(\Z) \text{ for all $x\in\Z$.}
\]
Likewise the identity
\begin{equation*}
\label{sandy2}
\Bc(-1)^2\lbarbrack x\rbarbrack^{-1}\lbarbrack x+1,-1,1,x+1,-1,1\rbarbrack\lbarbrack x
\rbarbrack = \lbarbrack x\rbarbrack^{-1}\lbarbrack
x,-1,1,x,-1,1\rbarbrack\lbarbrack x\rbarbrack
\end{equation*}
implies that
\begin{equation*}
\lbarbrack x+1,-1,1,x+1,-1,1\rbarbrack\in \BK''(\Z)\Longleftrightarrow
\lbarbrack
x,-1,1,x,-1,1\rbarbrack\in\BK''(\Z).
\end{equation*}
Since $\lbarbrack 1,-1,1,1,-1,1\rbarbrack =\Bc(1)^2\in\BK''(\Z)$, induction
shows that $\lbarbrack x,-1,1,x,-1,1\rbarbrack\in\BK''(\Z)$ for all
$x\in\Z$, establishing \eqref{dinner}.

We have $\BK(\Z)\leq \oFCF^{\,+}(\Z)$ since each of its generators
in \eqref{bitten} has even length, satisfying Definition \ref{diana}.
Since $\oFCF^{\, +}(\Z)$ is a normal subgroup of index $2$ with the
nontrivial coset of $\oFCF(\Z)/\oFCF^{\, +}(\Z)$ containing
$\lbarbrack 0\rbarbrack =\lbarbrack 0\rbarbrack^{-1}$, to show
that the normal closure of 
$\langle \Bc(1)^2,\Bc(-1)^2,\Bc(1)\Bc(-1)\Bc(1)\Bc(-1)\rangle$ in
$\oFCF(\Z)$ is equal to its normal closure in $\oFCF^{\, +}(\Z)$, it suffices
to observe that
\begin{align*}
\lbarbrack 0\rbarbrack \Bc(-1)^2\lbarbrack 0\rbarbrack &= \Bc(1)^{-2}\\
\lbarbrack 0\rbarbrack \Bc(1)^2\lbarbrack 0\rbarbrack &= \Bc(-1)^{-2}\\
\lbarbrack 0\rbarbrack \Bc(1)\Bc(-1)\Bc(1)\Bc(-1)\lbarbrack 0\rbarbrack
&=[\Bc(1)\Bc(-1)\Bc(1)\Bc(-1)]^{-1}.  \qedhere
\end{align*}
\end{proof}
\begin{theorem}[Cohn]
\label{rain}
The map $\overline{M}:\oFCF(\Z)\rightarrow \BBE^\pm(\Z)
=\SLTpm(\Z)$ is surjective with kernel
$\BK(\Z)$. Likewise $\overline{M}^{\, +}:\oFCF^{\, +}(\Z)\rightarrow \BBE(\Z)$
is surjective with kernel $\BK(\Z)$. 
\end{theorem}
\begin{proof}
Cohn \cite[Thm.~5.2]{Cohn} shows that discretely normed rings are universal for 
$\BBE^\pm$ as in Definition \ref{soak}, and $\Z$ is discretely normed.
The statement for $\overline{M}$
implies that for $\overline{M}^{\, +}$.
Of course, it follows from the Euclidean algorithm that $\BBE^{\pm}(\Z)=
\SLTpm(\Z)$; see, for example,\cite[Chap.~12, Sect.~4]{Artin}.
\end{proof}

Knowing the kernel of $\overline{M}^{\,+}:\oFCF^+(\OO)\cong \OO\ast\OO
\rightarrow \BBE(\OO)\leq\SLT(\OO)$ gives a presentation of $\BBE(\OO)$,
and a presentation of $\SLT(\OO)$ if $\overline{M}^{\, +}$
is surjective. More generally, information
on $\ker(\overline{M}^{\, +})$ gives information on the generation of 
$\SLT(\OO)$ if $\BBE(\OO)=\SLT(\OO)$.
It is an interesting question whether information on $\ker(\overline{M}^{\,+})$,
such as knowing $\BK(\OO)\leq \ker(\overline{M}^{\, +})$, gives 
results on bounded generation in case $\OO^\times$ is infinite 
as in \cite{mrs} and \cite{jz}.

We show that the presentation of $\SLT(\Z)$ given by Theorem \ref{rain} gives
the familiar amalgamated product presentation.
\begin{corollary}
\label{storm}
We have $\SLT(\Z)\cong \Z/6\Z\ast_{\Z/2\Z}\Z/4\Z$.
\end{corollary}
\begin{proof}
Note that by Remark \ref{feast}, $\Bc(-1)^{-1}=\lbarbrack 0,1,-1,1,0\rbarbrack$.
Let $\alpha=\lbarbrack 1,-1\rbarbrack\in\oFCF^+(\Z)$ and 
$\beta=\lbarbrack 1,-1,1,0\rbarbrack=\Bc(1)\lbarbrack 0\rbarbrack
\in\oFCF^+(\Z)$, so that
\begin{equation}
\label{red}
\alpha^3=\Bc(1)\Bc(-1),\, \beta^2=\Bc(1)\Bc(-1)^{-1},\,
\overline{M}(\alpha)=\begin{bmatrix}0 &1\\-1 & 1\end{bmatrix},\,
\text{ and }\overline{M}(\beta)=\begin{bmatrix}0 & 1 \\-1 & 0\end{bmatrix}.
\end{equation}

We now consider the group $\Z/6\Z\ast_{\Z/2\Z}\Z/4\Z$.  A presentation of
this group is the free group on two generators $\bm{\alpha}$
and $\bm{\beta}$ modulo the obvious relations:
\begin{equation}
\label{warm}
\Z/6\Z\ast_{\Z/2\Z}\Z/4\Z\cong \frac{\langle \bm{\alpha} \rangle\ast
\langle \bm{\beta}\rangle}{\ncl\langle \bm{\alpha}^6,
\bm{\alpha}^3\bm{\beta}^2,
\bm{\beta}^4\rangle},
\end{equation}
where $\ncl\langle \bm{\alpha}^6,
\bm{\alpha}^3\bm{\beta}^2,
\bm{\beta}^4\rangle$ is the normal closure 
of 
$\langle \bm{\alpha}^6,
\bm{\alpha}^3\bm{\beta}^2,
\bm{\beta}^4\rangle$ in $\langle \bm{\alpha} \rangle\ast
\langle \bm{\beta}\rangle$.

We now state and prove a series of claims:\\[.1in]
\textbf{Claim 1.} \textit{Set $\BK''(\Z)\colonequals
\ncl\langle \alpha^6, \alpha^3\beta^2,\beta^4\rangle$, the normal closure
of $\langle \alpha^6, \alpha^3\beta^2,\beta^4\rangle$ in $\oFCF^{\, +}(\Z)$.
Then $\BK''(\Z)$ is the normal closure of 
$\langle \Bc(1)^2, \Bc(-1)^2, \Bc(1)\Bc(-1)\Bc(1)\Bc(-1)
\rangle$ in $\oFCF^{\, +}(\Z)$.  So by Proposition \textup{\ref{forgot}},
$\BK''(\Z)=\BK(\Z)=\BK'(\Z)$.
} \\[.05in]
\textit{Proof.}
$\BK''(\Z)\subseteq \BK(\Z)$:
Using \eqref{red}, verify that $\overline{M}(\alpha^6)=
\overline{M}(\alpha^3\beta^2)=\overline{M}(\beta^4)=1\in\SLT(\Z)$.
Hence $\BK''(\Z)\subseteq \ker (\overline{M})
=\ker(\overline{M}^{\, +})=\BK(\Z)$ by Theorem \ref{rain}.\\
$\BK(\Z)\subseteq \BK''(\Z)$: It suffices to note the following:
\begin{align*}
\Bc(1)\Bc(-1)\Bc(1)\Bc(-1) &= \alpha^6,\\
\Bc(-1)^2  &=\beta^{-2}\alpha^3=(\alpha^3\beta^2)^{-1}\alpha^6,\\
\Bc(1)^2 &=\beta\alpha^{-3}\beta =\beta^{-1}[\beta^4(\alpha^3\beta^2)^{-1}]\beta\in\ncl\langle \alpha^6,\alpha^3\beta^2,\beta^4\rangle\equalscolon
\BK''(\Z),
\end{align*}
proving Claim 1.\\[.1in]
\textbf{Claim 2.} \textit{The elements $\alpha$ and $\beta$ generate
$\oFCF^{\, +}(\Z)$}.\\[.05in]
\textit{Proof.}
By Proposition \ref{goat}, $\oPCF^{\,+}(\Z)$ is a free group
on the two generators $\oBU(-1)=\oBU(1)^{-1}$ and $\oBL(-1)=\oBL(1)^{-1}$.
But
\begin{align*}
\oBU(-1)&=\lbarbrack -1,0\rbarbrack =\lbarbrack 0,0,-1,1,-1,0\rbarbrack
\lbarbrack 1,-1\rbarbrack = \beta^{-1}\alpha\\
\oBL(-1)&= \lbarbrack 0,-1\rbarbrack =\lbarbrack 0,0,-1,1,-1,0\rbarbrack
\lbarbrack 1,-1\rbarbrack \lbarbrack 1,-1\rbarbrack =\beta^{-1}\alpha^2,
\end{align*}
proving Claim 2.\\[.1in]
We can now conclude the proof of Corollary \ref{storm}.  By Theorem
\ref{rain} we have 
\begin{align*}
\SLT(\Z)&\cong \frac{\oFCF^{\, +}(\Z)}{\BK(\Z)}
\cong \frac{\langle\alpha\rangle \ast\langle \beta\rangle}
{\ncl\langle \alpha^6, \alpha^3\beta^2, \beta^4\rangle}
\quad\text{by Claims 1 and 2}\\
&\cong \Z/6\Z\ast_{\Z/2\Z}\Z/4\Z\quad\text{by \eqref{warm}}.\qedhere
\end{align*} 
\end{proof}

When $\OO$ is a discretely normed quadratic imaginary number ring,
Cohn shows $\BK(\OO)=\ker \overline{M}$ 
and, in \cite[Thm.~6.1]{Cohn},
$\BBEpm(\OO)\subset \SLTpm(\OO)$ is a proper containment.

\section{Periodic continued fractions}
\label{sec:PCF}

An infinite continued fraction $C$ is a formal expression $[c_1,c_2,\ldots]$
where the sequence of partial quotients does not terminate.
The infinite continued fraction $C=[c_1, c_2, \ldots]$ may
also be expressed as a nonterminating version of \eqref{contfrac}.
As in the finite case, the convergent $\CC_k(C)$ is the evaluation of 
$[c_1, c_2, \ldots, c_k].$
For $\MF$ as in \eqref{eq:FCFtoSLT},
set  $M_k(C)\colonequals \MF([c_1, c_2, \ldots, c_k])$
with  $M_0(C) =I$ so that $\CC_k(C)=M_k(C)_{11}/M_k(C)_{21}$.
The \textsf{value} (or \textsf{limit}) of $C$ is
\begin{equation*}
\hat{C}=\lim_{k\to\infty}\CC_k(C)
\end{equation*}
when the limit exists in $\PP^1(\BC)$, in which case 
$C$ \textsf{converges}.

A \textsf{periodic continued fraction} (PCF) $P$ is an
infinite continued fraction $[c_1,c_2,\ldots]$ together with a \textsf{type}
$(N,k)$ with $N\ge 0$, $k\ge 1$ such that $c_{n+k}=c_n$ for $n>N$.
We denote the sequence of partial quotients of the PCF $P$ by
\begin{equation}
\label{berate}
P=[b_1, \ldots, b_N, \overline{a_1, a_2, \ldots, a_k}\,].
\end{equation}
The natural number $k$ is the \textsf{period} of $P$; the 
\textsf{initial part} of $P$ in \eqref{berate}
is the FCF $\In(P)=[b_1,\ldots, b_N]$; and the \textsf{repeated part}
of $P$ is the FCF $\Rep(P)=[a_1, \ldots, a_k]$.
The PCF $P$ as in \eqref{berate} has Galois dual
\begin{equation}
\label{azure}
P^\ast \colonequals [b_1, \ldots, b_N,0,\overline{-a_k,\ldots, -a_1}\,];
\end{equation}
see, e.g., \cite[p.~384]{bej}.
For a ring $\OO\subseteq \BC$ containing 1,
say $P= [b_1,\ldots, b_N,\overline{a_1, a_2, \ldots, a_k}\,]$
is an $\OO$-$\PCF$ if $b_i, a_j\in\OO$
for $1\leq i\leq N$, $1\leq j\leq k$.
Let $\PCF(\OO)$ denote the set of all $\OO$-PCFs.  If $P\in\PCF(\OO)$, then
$P^\ast\in\PCF(\OO)$.

A PCF with type $(0,k)$ is purely periodic.
For brevity, we abbreviate purely periodic continued fraction as RCF, 
after the phrase \textsf{repeating continued fraction}. 
The set of RCFs which are $\OO$-PCFs is denoted $\RCF(\OO)$.
An \textsf{untyped periodic continued fraction} (UCF) $U$
is an infinite continued fraction which is periodic of some type $(N,k)$.
Notice that the type of a UCF is not uniquely determined:
A UCF $U$ with type $(N,k)$ also has type $(N',mk)$ for every $N'\ge N$
and every multiple $mk$ of $k$.

\begin{definition1}
\label{orange}
There are three equality relationships between a PCF $P$ of type $(N,k)$
and a PCF $P'$ of type $(N',k')$.
\begin{enumerate}
\item 
$P$ and $P'$ are \textsf{CF-equal} 
(written $=_{\textrm{CF}}$)
when $P$ and $P'$ are the same as UCFs.
\item 
\label{orange1}
$P$ and $P'$ are \textsf{k-equal} (written $=_{\textrm{k}}$) when $P$ and $P'$ are the same as UCFs
and $k=k'$.
\item 
$P$ and $P'$ are \textsf{equal} (written $=$) when $P$ and $P'$ are the same as PCFs, that is, $P=_{\textrm{CF}}P'$ and  $(N,k)=(N',k')$.
\end{enumerate}
\end{definition1}
\noindent For example, $[1,\overline{2}]=_{\textrm{CF}} [1,\overline{2,2}]$,
but
$[1,\overline{2}\,]\neq_{\textrm{k}} [1,\overline{2,2}\,]$. 
And $[1,\overline{2,1}\,]=_{\textrm{k}}
[1,2,\overline{1,2}\,]$, but
$[1,\overline{2,1}\,]\neq
[1,2,\overline{1,2}\,]$.

\subsection{An equivalence relation and group law on PCFs}
\label{sec:eqrelPCF}
We give an equivalence relation on $\PCF(\OO)$ and a group law
on the equivalence classes by means of the group law on $\oFCF(\OO)$.

\begin{definition1}
\label{def:PCFtoOFCF}
The map $\WP: \PCF(\OO)\rightarrow \Z/2\Z\ast\OO$ maps the
PCF $P$ with initial part $\In(P)$ and repeating part $\Rep(P)$ to
\begin{equation}
\label{teal}
\WP(P) = \WF(\In(P))\WF(\Rep(P))\WF(\In(P))^{-1}=
 \WF(\In(P))\WF(\Rep(P))\WF(\In(P)^\ast),
\end{equation}
where $\WF$ is as in Definition \ref{pork} and $\In(P)^\ast$
is as in Remark \ref{feast}.
\end{definition1}

The type $(N,k)$ is essential data attached to $P$ for computing 
the map $\WP$.
We also have maps
$\SF:\PCF(\OO)\rightarrow \FCF(\OO)$ and
$\SR:\FCF(\OO)\rightarrow \RCF(\OO)\subseteq \PCF(\OO)$ defined by
\begin{align}
\label{eq:FtoR}
\SF([b_1,\ldots, b_N,\overline{a_1,\ldots , a_k}\,]) & =\begin{cases}
[b_1,\ldots, b_N,a_1, \ldots, a_k,0,-b_N, \ldots, -b_1,0], \ N > 0\\
[a_1,\ldots , a_k],\ N=0.
\end{cases}\\
\label{eq:RtoF}
\SR([c_1,\ldots, c_n]) &= [\,\overline{c_1,\ldots, c_n}\,].
\end{align}
Note that 
$\SF\circ\SR$ is the identity on $\FCF(\OO)$ and
the map $\WP$ factors as 
\begin{equation}
\label{settle}
\WP=\WF\circ\SF.
\end{equation}

The maps $\SR$ and 
$\SF|_{\RCF(\OO)}$ give inverse bijections:
\begin{equation}
\label{rated}
\SF|_{\RCF(\OO)}:\RCF(\OO)\stackrel{\simeq}{\longrightarrow}\FCF(\OO)
\quad\text{and}\quad\SR:\FCF(\OO)\stackrel{\simeq}{\longrightarrow}
\RCF(\OO).
\end{equation}

The equivalence relation $\sim$ on $\FCF(\OO)$
defined in Section \ref{harbor}
induces an equivalence relation, also denoted $\sim$.
on $\PCF(\OO)$.

\begin{definition1}
\label{wealth}
For $P, P'\in\PCF(\OO)$, say
\begin{align*}
P\sim P'&\Longleftrightarrow \SF(P)\sim\SF(P')\Longleftrightarrow \WF(\SF(P))=\WF(\SF(P'))\\ \nonumber
&\Longleftrightarrow
\WP(P)=\WP(P')\in\Z/2\Z\ast\OO .
\end{align*}

\end{definition1}

\begin{remark1}
\label{wealth2}
For $F, F'\in\FCF(\OO)$ we have 
\begin{equation}
F\sim F'\Longleftrightarrow
\SR(F)\sim\SR(F')\text{ since  }\SF\left(\SR(F)\right)=F.
\end{equation}
\end{remark1}

\begin{proposition}
\label{desktop}
If $P, P'\in\PCF(\OO)$ and $P=_{\textrm{k}}P'$ as in
 Definition \textup{\ref{orange}\eqref{orange1}}, then $P\sim P'$.
\end{proposition}

\begin{proof}
First consider the case that $P$ is of type $(N,k)$ and $P'$
is of type $(N+1,k)$ with $P=_{\textrm{k}}P'$.  Then we must have
$$
P=[b_1,\ldots, b_N,\overline{a_1,\ldots, a_k}]\quad\text{and}\quad
P'=[b_1,\ldots, b_N, a_1,\overline{a_2, \ldots, a_k, a_1}].$$
We have
\begin{align*}
\SF(P)&=[b_1,\ldots, b_N, a_1, \ldots, a_k,0,-b_N,\ldots, -b_1,0]\text{ and}\\
\SF(P')&=[b_1,\dots, b_N,a_1,a_2,\ldots , a_k, a_1,0,-a_1,-b_N,\ldots, -b_1,0]
\end{align*}
with $\SF$ as in \eqref{eq:FtoR}.
But then $\SF(P)\sim \SF(P')$ by applying \eqref{eq:Dsinglerelation}.

The proposition then follows from this by iteration.
\end{proof}

\begin{definition1}
An RCF  $R=[\,\overline{a_1,\ldots a_n}\,]\in\RCF(\OO)$ 
is \textsf{reduced} when $\SF(R)\in\FCF(\OO)$
is reduced as in Definition \ref{stunt}, i.e., 
$a_i\neq 0$ for $2\leq i\leq n-1$.
\end{definition1}

\begin{remark1}
\label{ship}
Note that $F\in\FCF(\OO)$ is reduced if and only
if $\SR(F)\in\RCF(\OO)$ is reduced since $\SF\circ\SR(F)=F$.
\end{remark1}

\begin{proposition}
\label{prop:reducedPCF}
For $P\in\PCF(\OO)$, 
let $\SF(P)_\mathrm{red}\in\FCF(\OO)$ be as in Proposition 
\textup{\ref{prop:reducedFCF}}. Then 
the element $P_\mathrm{red} = \SR(\SF(P)_{\mathrm{red}})$ is
the unique reduced element of $\RCF(\OO)$ such that $P_\mathrm{red} \sim P.$
\end{proposition}

\begin{proof}
For $P\in\PCF(\OO)$ we have that
$$P_\mathrm{red} = \SR\left(\SF(P)_{\mathrm{red}}\right)\in \RCF(\OO)$$
is reduced, as stated in Remark \ref{ship}.  We claim that
$P\sim P_\mathrm{red}$.  Note that
$\SF(P)\sim\SF(P)_{\mathrm{red}}$ implies that
$\SR\left(\SF(P)\right)\sim\SR\left(\SF(P)_{\mathrm{red}}\right)$ as in Remark
\ref{wealth2}. By Definition \ref{wealth}, $\SF(P)=\SF(\SR(\SF(P)))$ implies
$P\sim \SR(\SF(P))$.  Thus,
$P\sim \SR\left(\SF(P)_{\mathrm{red}}\right)$.
Suppose $R, R'\in\RCF(\OO)$ are both reduced and $R\sim R'$.
In this case, $\SF(R),\SF(R')\in\FCF(\OO)$ are both reduced and $\SF(R)\sim\SF(R')$
by Definition \ref{wealth}.  Proposition \ref{prop:reducedFCF} shows
that $\SF(R)=\SF(R')\in\FCF(\OO)$ and hence $R=R'$ by \eqref{rated}.
\end{proof}

Write $\oPCF(\OO)\colonequals \PCF(\OO)/\sim $.  
For $P=[b_1,\ldots, b_N,\overline{a_1,\ldots , a_k}\,]\in\PCF(\OO)$,
let $\overline{P}=\lbarbrack b_1,\ldots, b_N,\overline{a_1, \ldots
, a_k}\,\rbarbrack\in\oPCF(\OO)$ be the
equivalence class containing $P$. By Definition \ref{wealth} and 
Remark \ref{wealth2}, 
the maps
$\SF$ and $\SR$ induce inverse bijections
\begin{equation*}
\overline{\SF}:\oPCF(\OO)\stackrel{\simeq}{\longrightarrow}\oFCF(\OO)
\quad\text{and}\quad \overline{\SR}:\oFCF(\OO)
\stackrel{\simeq}{\longrightarrow}\oPCF(\OO).
\end{equation*}
Via the bijections $\overline{\SF}$, $\overline{\SR}$ 
the binary operation $\star$ of
Proposition \ref{mush} on $\oFCF(\OO)$ induces a binary
operation, also denoted $\star$, on $\oPCF(\OO)$:
for $P, P'\in\PCF(\OO)$ we have
\begin{equation}
\label{status}
\overline{\SF}(\overline{P}\star \overline{P}')=
\overline{\SF}(\overline{P})\star
\overline{\SF}(\overline{P}')\quad\text{and}\quad
\overline{P}\star \overline{P}'=\overline{\SR}
\left(\overline{\SF}(\overline{P})\star
\overline{\SF}(\overline{P}')\right).
\end{equation}

\begin{theorem}
\label{thm:PCFisom}
The semigroup $\oPCF(\OO)$
with its binary operation $\star$ 
is a group isomorphic to $\oFCF(\OO)\cong \Z/2\Z\ast\OO$.
The map $\WP:\PCF(\OO)\rightarrow \Z/2\Z\ast\OO$
of Definition \textup{\ref{def:PCFtoOFCF}}  induces an 
isomorphism of groups  $\oWP:\oPCF(\OO)\stackrel{\simeq}
{\longrightarrow}\Z/2\Z\ast\OO$.
\end{theorem}

For the convenience of the reader we give in Figure 1  the maps defined
thus far and the relations between them.
\begin{figure}[ht]
\begin{equation*}
\xymatrixcolsep{3.5pc}
\xymatrixrowsep{3pc}
\xymatrix{
\PCF(\OO)\ar@<0ex>[r]^{\SF}\ar@<0ex>[d]\ar@/^4pc/[rrrd]^{E}   & 
\FCF(\OO)\ar@<0ex>[d]
\ar@/^1pc/[rrd]^{M}\ar@<0ex>[rd]^{W_{\sF}}
& & &\\
\oPCF(\OO)\ar@<0ex>[r]^{\overline{\SF}}_{\approx}
\ar@/_3.4pc/[rrr]_{\overline{E}}
&\oFCF(\OO)\ar@<0ex>[r]_{\approx}
^{\overline{W}_{\sF}}\ar@/_1.6pc/[rr]_{\overline{M}}
& \Z/2\Z\ast \OO\ar@<0ex>[r]^{\BM}&
\BBEpm(\OO)\ar@{^{(}->}[r]&\SLTpm(\OO)
}
\end{equation*}
\vspace*{.1in}
\begin{equation*}
\xymatrixcolsep{3.5pc}
\xymatrixrowsep{3pc}
\xymatrix{
\oPCF(\OO)\ar@<0ex>[r]^{\overline{\SF}}_{\approx}
\ar@/^4pc/[rrr]^{\overline{E}}
   & \oFCF(\OO)
\ar@/^1.8pc/[rr]^{\overline{M}}\ar@<0ex>[r]^{\overline{W}_{\sF}}_{\approx}
& \Z/2\Z\ast\OO\ar@{->>}[r]^{\BM}&\BBEpm(\OO)\ar@{^{(}->}[r]&\SLTpm(\OO)\\
\oPCF^{\, +}(\OO)\ar@<0ex>[r]_{\approx}\ar@{^{(}->}[u]
& \oFCF^{\, +}(\OO)
\ar@<0ex>[r]^{\overline{W}_{\sF^+}}_{\approx}\ar@{^{(}->}[u]
\ar@/_1.6pc/[rr]_{\overline{M}^{\, +}}
& \OO\ast\OO\ar@{->>}[r]\ar@{^{(}->}[u]_{\psi}
&\BBE(\OO)\ar@{^{(}->}[r]\ar@{^{(}->}[u]
&\SLT(\OO)\ar@{^{(}->}[u]   
}
\end{equation*}
\caption{The commutative diagrams formed by the maps in Sections 
\ref{sec:contfrac}
and \ref{sec:eqrelPCF}. The maps $W_{\sP}=W_{\sF}\circ\SF:
\PCF(\OO)\rightarrow \Z/2\Z\ast\OO$ and $\overline{W}_{\sP}=
\overline{W}_{\sF}\circ\overline{\SF}:\oPCF(\OO)\rightarrow \Z/2\ast\OO$ can
be added to the diagrams.}
\end{figure}

\begin{remark1}
If $\overline{P}\in\oPCF(\OO)$ with $P\in\PCF(\OO)$, then
Proposition \ref{prop:reducedPCF}
 shows that the equivalence class $\overline{P}$
contains a unique reduced RCF $P_{\textrm{red}}=\overline{P}_{\textrm{red}}$.  We call the $P_{\textrm{red}}=\overline{P}_{\textrm{red}}$
the \textsf{normal form} of $\overline{P}$.  Two elements 
$\overline{P}, \overline{P}'\in\oPCF(\OO)$ are equal if and only if they have
the same normal forms: 
$$\overline{P}=\overline{P}'\Longleftrightarrow \overline{P}_\textrm{red}=
\overline{P}'_\textrm{red}\in
\RCF(\OO).$$

For $\OO \ne \BC$, there are PCFs $P$ with not all partial quotients in $\OO$
with normal form $P_{\mathrm{red}}\in\RCF(\OO)$; most simply,
$[\,\overline{c,0,-c}\,] \sim [\,\overline{0}\,]$ for any $c \in \BC$.
We test the equivalence of $P, P'\in\PCF(\OO)$ via their
normal forms $P_{\mathrm{red}}, P'_{\mathrm{red}}\in\RCF(\OO)$.
It is never necessary to make an excursion outside $\OO$ to determine
equivalence of PCFs with partial quotients in $\OO$.

We give an example of how to compute $P_{\mathrm{red}}$ using
Proposition \ref{prop:reducedPCF} and Remark \ref{gifted}.
Suppose $P=[3,0,-3,4,\overline{0,1,0,-5}\,]\in\PCF(\Z)$.
Then \[\SF(P)=[3,0,-3,4,0,1,0,-5,0,-4,3,0,-3,0]\in\FCF(\Z).\]
We have
$$
\SF(P)\sim[0,4,0,1,0,-5,0,-4,3,0,-3,0]\sim [0,5,0,-9,0,0]\sim
[0,-4,0,0]\sim [0,-4].$$
Hence $\SF(P)_{\mathrm{red}}=[0,-4]\in\FCF(\Z)$ and 
$P_{\mathrm{red}}=\SR([0,-4])=[\,\overline{0,-4}\,]\in\RCF(\Z)$ with
$$
\lbarbrack 3,0,-3,4,\overline{0,1,0,-5}\,\rbarbrack=
\overline{P}=\overline{P}_{\mathrm{red}}=\lbarbrack\,\overline{0,-4}\,\rbarbrack\in\oPCF(\Z).
$$
In practice, $\overline{P}\star\overline{P}'\in\oPCF(\OO)$ is computed
using the normal forms $\overline{P}_{\textrm{red}}$,$\overline{P}'_{\textrm{red}}$:
if $\overline{P}_{\textrm{red}}=[\,\overline{c_1, \ldots, c_n}\,]\in\RCF(\OO)$ and
$\overline{P}'_{\textrm{red}}=[\,\overline{c'_1,\ldots, c'_{n'}}\,]\in\RCF(\OO)$, then
\[
\overline{P}\star\overline{P}'=\lbarbrack\, \overline{c_1, \ldots,  c_n, c'_1,
\ldots , c'_{n'}}\,\rbarbrack\in\oPCF(\OO).
\]
\end{remark1}

\begin{remark1}
\label{roses}
Suppose 
\[
P=[ b_1,\overline{a_1,\ldots, a_k}\,],\,
P'=[ b_1',\overline{a_1',\ldots, a_{k'}'}\,] \in\PCF(\OO)
\]
are of types $(1,k)$ and $(1,k')$ respectively.  Then the product takes
a particularly simple form with a representative of type $(1,k+k')$:
\begin{align}
\label{meter}
\overline{P}\star \overline{P'}&= \lbarbrack \,\overline{b_1, a_1,\ldots,
a_k,0,-b_1,0,b_1', a_1',\ldots a_{k'}', 0, -b_{1}',0}\,\rbarbrack\text{ by
\eqref{status}}\nonumber\\
&= \lbarbrack \,
\overline{b_1, a_1,\ldots, a_k+b_1'-b_1, a_1', \ldots, a_{k'}'-b_1',0}\,
\rbarbrack\text{ using \eqref{muster}} \nonumber \\  
&=\lbarbrack b_1,\overline{a_1,\ldots, a_{k-1}, a_k+b'_1-b_1, a'_1,
\ldots, a'_{k'-1}, a'_{k'}-b_1',0,b_1}\,\rbarbrack\nonumber\\
&= \lbarbrack b_1, \overline{a_1,\ldots, a_{k-1}, a_{k}+b_1'-b_1, a_1', \ldots,
a'_{k'-1}, a'_{k'}-b'_1+b_1}\,\rbarbrack .
\end{align}

A similar exercise shows that the product is simple for $P,P'\in\PCF(\OO)$
with $\In(P)=\In(P')$.
Suppose 
\begin{equation*}
P=[b_1,\ldots, b_N,\overline{a_1,\ldots, a_k}\,]\text{ and }
P'=[b_1,\ldots, b_N,\overline{a_1', \ldots , a_{k'}'}\,].
\end{equation*}
Then
\begin{equation}
\label{beetle}
\overline{P}\star\overline{P'}=\lbarbrack
b_1,\ldots, b_N,\overline{a_1,\ldots, a_k,a_1',\ldots, a_{k'}'}\,\rbarbrack.
\end{equation}

\end{remark1}

\begin{remark1}
\label{angel}
For every PCF of the form $P=[b_1,\ldots, b_N,\overline{0}\,]$, 
$\WP(P) \in \Z/2\Z\ast \OO$ is an involution, but, in
general, $P$ is not equivalent to $[\,\overline{0}\,]$.

Every PCF of the form $[b_1,\ldots, b_N,\overline{0,0}\,]$ is in the class of the identity in
$\oPCF(\OO)$, that is, $P \sim [\,\overline{0,0}\,]$.
\end{remark1}

\section{The subgroup \texorpdfstring{$\oPCF_{\OO}(A)\leq\oPCF(\OO)$}{\textover PCF\textoinferior(A)\textle\textover PCF(O)} and its representations}
\label{sec:twoDreps}

For $P\in\PCF(\OO)$, \cite[Defn.~2.4]{bej}
defines a matrix $E(P)$ which plays a large role
in the study of $P$.  We review the definition, using
the notation of this paper.

\begin{definition1}
\label{ghost}
Let $E:\PCF(\OO)\rightarrow \SLTpm(\OO)$ be the map
$E=\MF\circ\SF$
with $\MF$ as in \eqref{eq:FCFtoSLT} and $\SF$ as in 
\eqref{eq:FtoR}.
\end{definition1}
We denote the quadratic polynomial $\Quad(E(P))$
as in \eqref{recycling}  by $\Quad(P)\in\OO[x]$ 
and the multiset of its roots $\Roots(E(P))$ by $\Roots(P)$.

\begin{proposition}
\label{prop:EmatchesM}
The map $E$ induces a homomorphism $\overline{E}:\oPCF(\OO)
\rightarrow\SLTpm(\OO)$ with image
$\BBEpm(\OO)\le\SLTpm(\OO)$ as in Section \textup{\ref{sec:contfrac}}.
\end{proposition}

\begin{proof}
From Definition \ref{wealth} and Proposition \ref{prop:mapSLT}
we see that $E$ is well-defined on equivalence classes in $\PCF(\OO)$.
By Proposition \ref{planet}, the image of $E$ is $\BBEpm(\OO)$.
\end{proof}

\begin{definition1}
\label{def:vecPCF}
Let $\beta \in \PP^1(\BC).$  Set
$$\oPCF_\OO(\beta) = 
\{ \overline{P} \in \oPCF(\OO) \mid \overline{E}(\overline{P})\beta = \beta\}.
$$
\end{definition1}
Recall from Proposition \ref{prop:singlevector}
that $T(\beta)$ is the subgroup $\{M\in\GLT(\BC)\mid M\beta=\beta\}$
and $T(\beta)$ has a multiplicative character $\lambda_\beta$
mapping $M$ to its $v(\beta)$-eigenvalue.
The preimage of $T(\beta) \cap \BBEpm(\OO)$ under $\overline{E}$ is $\oPCF_\OO(\beta)$,
so $\oPCF_\OO(\beta)$ is a subgroup of $\oPCF(\OO)$. The 
multiplicative character $\lambda_\beta \circ \overline{E}$ maps an element
$\overline{P}\in\oPCF_\OO(\beta)$ to the $v(\beta)$-eigenvalue of $\overline{E}(\overline{P})$.

\begin{definition1}
\label{roster}
\begin{enumerate}
\item
For $Q\in\OO[X]$ with $\deg Q \leq 2,$ set
\[
\oPCF_{\OO}(Q)=\{\overline{P}\in \oPCF(\OO)\mid \overline{E}(\overline{P})\in G(Q)\},
\]
where $G(Q)$ is the subgroup of $\GLT(\BC)$ defined in
Proposition \ref{prop:Qgroup}.
\item
For $A\in\GLT(\OO)$, set
$$\oPCF_{\OO}(A) = \oPCF_{\OO}(\Quad(A)).$$
\item
\label{def:matrixPCF}
Recall from Proposition \ref{prop:EmatchesM}
that $E:\PCF(\OO)\rightarrow \SLTpm(\OO)$ factors through $\oPCF(\OO)$:
\[
E:\PCF(\OO)\longrightarrow \oPCF(\OO)\stackrel{\overline{E}}{\longrightarrow}
\SLTpm(\OO).
\]
For $P \in \PCF(\OO)$, set
$$\oPCF_{\OO}(P) = \oPCF_{\OO}(\overline{P})=\oPCF_{\OO}(E(P))=\oPCF_{\OO}
(\overline{E}(\overline{P})).$$
\end{enumerate}
\end{definition1}
For $Q\in \OO[X]$ with $\deg Q \leq 2$, the preimage of $G(Q) \cap \BBEpm(\OO)$ under $\overline{E}$ is $\oPCF_\OO(Q)$,
so $\oPCF_\OO(Q)$ is a group. For $\beta\in\Roots(Q)$, the group $\oPCF_\OO(Q)$ 
has a multiplicative character $\lambda_\beta \circ \overline{E}$
inherited from the containment $\oPCF_\OO(Q) \subset \oPCF_\OO(\beta)$.

\begin{theorem}
\label{thm:characters}
Let $A \in \GLT(\OO)$.
\begin{enumerate}
\item
\label{charA}
If $\Roots(A) = \{\beta, \beta^*\}$, $\beta \neq \beta^*$,
and $A$ has corresponding eigenvectors $v(\beta), v(\beta^*)$, then
$$\oPCF_{\OO}(A) = \oPCF_\OO(\beta) \cap \oPCF_\OO(\beta^*),$$
and the values of the characters $\lambda_\beta \circ \overline{E}$ and
$\lambda_{\beta^*}\circ\, \overline{E}$ on $\oPCF_\OO(A)$ are
quadratic integral over $\OO$ and in $\OO[\beta]^\times$.
The product of these characters is $\det \circ\, \overline{E}$.
\item
\label{charB}
If $\Roots(A)$ has one element $\beta$ of multiplicity $2$, $\oPCF_\OO(A)$
has a character
$\lambda_\beta \circ\, \overline{E}$ whose square is $\det \circ\, \overline{E}$.
\item
\label{charC}
Otherwise, $A$ is a scalar multiple of the identity and $\oPCF_\OO(A)$ has a
character
$\lambda \circ\, \overline{E}$ equal to $\lambda_\beta \circ \overline{E}$
for every $\beta.$  Its square is $\det \circ\, \overline{E}$.
\end{enumerate}
\end{theorem}

\begin{proof}
Let $\overline{P} \in \oPCF(A)$ and $E = \overline{E}(\overline{P}) \in \BBEpm(\OO)$.

\eqref{charA}:
If $\Quad(A)$ is not zero and not a square, then
$\Roots(A)$ has distinct elements $\beta, \beta^*$ and $E$ has corresponding eigenvalues
$\lambda_\beta(E), \lambda_{\beta^*}(E)$ that solve the monic quadratic equation $\det(E-\lambda I)$.
It follows from Propostion \ref{prop:Qgroup} that
$$\oPCF_{\OO}(A) = \oPCF_\OO(\beta) \cap \oPCF_\OO(\beta^*).$$
Relation \eqref{thunder} shows
the eigenvalue $\lambda_\beta(E) \in \OO[\beta]$.
The sum $\lambda_\beta(E)+\lambda_{\beta^*}(E)$ is in $\OO$, so $\lambda_{\beta^*}(E)\in\OO[\beta]$.
The product $\lambda_\beta(E)\lambda_{\beta^*}(E)$ is $\det E =  \pm1$, so both 
$\lambda_\beta(E)$ and $\lambda_{\beta^*}(E)$ are units.

\eqref{charB}:
When $\Quad(A)$ is a nonzero square, $A$ has a single eigenvector $v(\beta)$.
For $E$ as above,
the eigenvalue $\lambda_\beta(E)$ solves a monic linear equation over $\OO$ and
$\lambda_\beta(E)^2 = \pm1$.

\eqref{charC}:
When $\Quad(A)$ is zero, every $E \in \oPCF_\OO(A)$ is $\mu_E I$, a scalar
multiple of $I$.  In this case, $\lambda_\beta(E) = \mu_E$ for every $\beta$.
\end{proof}

\begin{remark1}
Let $P=[b_1,\overline{a_1,\ldots,a_k}]$ be a PCF of type $(1,k)$.
The convergent $\mathcal{C}_k(P)$
relates $\Roots(P)$ and the eigenvalues of $E = \MF(\SF(P))$.
When
\begin{equation}
\label{eq:type1Kspecial}
P \sim [\overline{b_1, a_1, \ldots, a_{k-1}, a_k-b_1,0}],
\end{equation}
we have $E(P)=\MF([b_1,a_1, \ldots, a_{k-1}, a_k-b_1,0])$.
For $M=M([b_1,a_1,\ldots,a_{k-1}])=\left[\begin{smallmatrix} m_{11} & 
m_{12}\\m_{21} & m_{22}\end{smallmatrix}\right]$,
we have
\[
E=\begin{bmatrix}
E_{11} & E_{12}\\E_{21} & E_{22}\end{bmatrix}=
\begin{bmatrix}
m_{11} & m_{12}\\ m_{21} & m_{22}\end{bmatrix}
\begin{bmatrix}
1 & a_k-b_1\\
0 & 1 \end{bmatrix}=
\begin{bmatrix}
m_{11} & m_{11}(a_k-b_1)+m_{12}\\
m_{21} & m_{21}(a_k - b_1) + m_{22}\end{bmatrix}.
\]
Thus, $E_{11}/E_{21}$ is the $k$th convergent $\mathcal{C}_k(P)$.
If $\Roots(P)$ has two distinct elements $\beta, \beta^*$ and $E$ has associated eigenvalues
$\lambda, \lambda^*$, 
$$\lambda=E_{11}-E_{21}\beta^\ast.$$
If $\Roots(P)$ has one element $\beta$ of multiplicity two, $E_{11}-E_{21}\beta$ is the eigenvalue of $E$.  If $E$ is a scalar multiple of $I$, its eigenvalue is $E_{11}$.
\end{remark1}
\begin{remark1}
The maps defined in this section form a commutative diagram as given in
Figure 2.
\begin{figure}[ht]
\begin{equation*}
\label{snails3}
\xymatrixcolsep{3.5pc}
\xymatrixrowsep{2.5pc}
\xymatrix{
\oPCF(\OO)\ar@<0ex>[r]^{\overline{E}}
   & \BBEpm(\OO)\ar@{^{(}->}[r]   & \SLTpm(\OO)\\
\oPCF_\OO(\beta)\ar@<0ex>[r]\ar@{^{(}->}[u]
   & T(\beta)\cap\BBEpm(\OO)\ar@<0ex>[r]^{\lambda_\beta}\ar@{^{(}->}[u]
& \OO[\beta]^\times\\
\oPCF_\OO(Q)\ar@<0ex>[r]\ar@{^{(}->}[u]
   & G(Q)\cap\BBEpm(\OO)\ar@{^{(}->}[u]
}
\end{equation*}
\caption{The commutative diagram formed by the maps in Section 
\ref{sec:twoDreps}
for $Q\in\OO[X]$ with $\deg(Q)\leq 2$ and $\beta\in\Roots(Q)$.}
\end{figure}
\end{remark1}

\section{Examples of Theorem \ref{thm:characters} arising from number rings}
\label{dirge}

\subsection{\texorpdfstring{\except{toc}{\boldmath{$P_1=
[1,\overline{2}], \hat{P}_1=\sqrt{2}$}}\for{toc}{$P_1=
[1,\overline{2}], \hat{P}_1=\sqrt{2}$}}
{P\textoneinferior=[1,\textover{2}], \^{P}\textoneinferior=\textsurd 2}}
\label{runt}
The $\Z$-PCF $P_1$ is of type $(N,k)=(1,1)$ and
\[
E(P_1)=D(1)D(2)D(1)^{-1}=\begin{bmatrix}
1 & 2 \\ 1 & 1\end{bmatrix},
\]
so $\Quad(P_1)=x^2-2$, $\Roots(P_1)=\{ \sqrt{2},-\sqrt{2}\}$, and the first convergent
$\mathcal{C}_1(P_1) =1/1$.  
We have $\overline{P}_1\in\oPCF_\Z(P_1)$, with $\oPCF_\Z(P_1)$
as in Definition \ref{roster}\eqref{def:matrixPCF}.
The characters 
\begin{equation}
\label{diamond}
\lambda\colonequals \lambda_{\sqrt{2}}\circ \overline{E},\,\, 
\lambda^*\colonequals \lambda_{-\sqrt{2}}\circ \overline{E} \,\,
\text{ of $\oPCF_\Z(P_1)$}
\end{equation}
map $\overline{P}_1$
to the eigenvalues of $E(P_1)$:
\begin{equation}
\label{took}
\lambda(\overline{P}_1)=1+\sqrt{2}\text{ and }
\lambda^\ast(\overline{P}_1)=1-\sqrt{2},
\end{equation}
which are units in $\Z[\sqrt{2}]$,
and $\lambda(\overline{P}_1)\lambda^\ast(\overline{P}_1)
=\det(E({P}_1))=(-1)^k=-1$.

\subsection{\texorpdfstring{\except{toc}{\boldmath{$P_2=
[1,\overline{2,2,2}], \hat{P}_2=\sqrt{2}$}}\for{toc}{$P_2=
[1,\overline{2,2,2}], \hat{P}_2=\sqrt{2}$}}
{P\texttwoinferior=[1,\textover 2,\textover 2,\textover 2], \^{P}\texttwoinferior=\textsurd 2}}
\label{runt1}

The $\Z$-PCF $P_2$ has type $(N,k)=(1,3)$, and
\begin{align*}
E(P_2)&=D(1)D(2)^3D(1)^{-1}=\begin{bmatrix} 7 & 10\\5 & 7\end{bmatrix}=
E(P_1)^3=\begin{bmatrix} 1 & 2 \\ 1 & 1\end{bmatrix}^3,
\end{align*}
so $\Quad(P_2)=5x^2-10 = 5\Quad(P_1)$ and $\Roots(P_2)=\Roots(P_1)=
\{\sqrt{2}, -\sqrt{2}\}$.  By \eqref{beetle}, 
the equivalence class 
$\overline{P}_2$ satisfies 
$$\overline{P}_2 = \overline{P}_1\star\overline{P}_1
\star\overline{P}_1 \in \oPCF_\Z(P_2)=\oPCF_\Z(P_1).$$
The convergent $\mathcal{C}_3(P_2)$ is 7/5.
The characters $\lambda$, $\lambda^\ast$ of $\oPCF_\Z(P_2)=\oPCF_\Z(P_1)$
in \eqref{diamond} map
$\overline{P}_2$ to the eigenvalues of $E(P_2)$:
\[\lambda(\overline{P}_2)=7+5\sqrt{2}\text{ and }
\lambda^\ast(\overline{P}_2)=7-5\sqrt{2},\]
which are units in $\Z[\sqrt{2}]$,
and $\lambda(\overline{P}_2)\lambda^\ast(\overline{P}_2)
=\det(\overline{E}(\overline{P}_2))=(-1)^3=-1$.
Consequent to $\lambda$
being a homomorphism, 
\[
\lambda(\overline{P}_2)=\lambda(\overline{P}_1\star 
\overline{P}_1\star \overline{P}_1)=5\sqrt{2}+7=
(1+\sqrt{2})^3=\lambda(\overline{P}_1)^3
\]
using \eqref{took}.

\subsection{\texorpdfstring{\except{toc}{\boldmath{$P_3=
[2,\overline{-2,4}], \hat{P}_3=\sqrt{2}$}}\for{toc}{$P_3=
[2,\overline{-2,4}], \hat{P}_3=\sqrt{2}$}}
{P\textthreeinferior=[2,\textover -2,\textover 4], 
\^{P}\textthreeinferior=\textsurd 2}}
\label{runt2}

The $\Z$-PCF $P_3$ is of type $(N,k)=(1,2)$ and
\begin{align*}
E(P_3) & = D(2)D(-2)D(4)D(2)^{-1}=\begin{bmatrix}
-3 & -4\\ -2 & -3\end{bmatrix},
\end{align*}
so $\Quad(P_3)=-2x^2+4 = -2\Quad(P_1)$ and $\Roots(P_3)=\Roots(P_1)$.
The convergent $\mathcal{C}_2(P_3)$ is $3/2$.
The class  $\overline{P}_3$ is in $\oPCF_\Z(P_1)=\oPCF_\Z(P_3)$.
For the characters $\lambda$, $\lambda^\ast$ as in \eqref{diamond},
we have 
\begin{equation}
\label{peets}
\lambda(\overline{P}_3)=-3-2\sqrt{2}\text{ and }
\lambda^\ast(\overline{P}_3)=-3+2\sqrt{2},
\end{equation}
with
which are units in $\Z[\sqrt{2}]$,
and $\lambda(\overline{P}_3)\lambda^\ast(\overline{P}_3)=
\det(\overline{E}(\overline{P}_3))=(-1)^2=1$.

\subsection{\texorpdfstring{\except{toc}{\boldmath{$P_4=
[2,\overline{3,-2,3}], \hat{P}_4=\sqrt{2}$}}\for{toc}{$P_4=
[2,\overline{3,-2,3}], \hat{P}_4=\sqrt{2}$}}
{P\textfourinferior=[2,\textover 3,\textover -2,\textover 3], 
\^{P}\textfourinferior=\textsurd 2}}
\label{runt3}

The $\Z$-PCF
$P_4$ is of type $(N,k)=(1,3)$ and
\begin{align*}
E(P_4) &= D(1)D(3)D(-2)D(3)D(1)^{-1}=\begin{bmatrix} -7 & -10\\
-5 & -7 \end{bmatrix},
\end{align*}
so $\Quad(P_4)=-5x^2+10=-5\Quad(P_1)$ and $\Roots(P_4)=\Roots(P_1)$.
The convergent $\mathcal{C}_3(P_4)$ is $7/5$.
We have  $\overline{P}_4\in\oPCF_Z(P_4)=\oPCF_\Z(P_1)$ and, by \eqref{meter},
\[
\overline{P}_4=\lbarbrack 1, \overline{2-1+2, -2, 4-2+1}\rbarbrack
 = \lbarbrack 1, \overline{2}\rbarbrack  \lbarbrack 2, \overline{-2,4}
\rbarbrack =
\overline{P}_1\star\overline{P}_3.
\]
The values of the characters $\lambda, \lambda^\ast$ in \eqref{diamond}
on $\overline{P}_4$ are
$$\lambda(\overline{P}_4)=-7-5\sqrt{2}\text{ and }
\lambda^\ast(\overline{P}_4)=-7+5\sqrt{2}.$$
We have
$\lambda(\overline{P}_4)\lambda^\ast(\overline{P}_4)=
\det(\overline{E}(\overline{P}_4))=(-1)^3=-1.$
From Example \ref{runt} $\lambda(\overline{P}_1)=1+\sqrt{2}$ and from
Example \ref{runt2} $\lambda(\overline{P}_3)=-2\sqrt{2}-3$.
We verify using \eqref{took} and \eqref{peets}
\[
\lambda(\overline{P}_4)=\lambda(\overline{P}_1\star \overline{P}_3)
=-5\sqrt{2}-7=(1+\sqrt{2})(-2\sqrt{2}-3)
=\lambda(\overline{P}_1)\lambda(\overline{P}_3).
\]

\subsection{\texorpdfstring{\except{toc}{\boldmath{$P_5=[442+312\sqrt{2}, \overline{-298532+211094\sqrt{2},
884+624\sqrt{2}}],\hat{P}_5=\sqrt{2+\sqrt{2}}$}}\for{toc}{$P_5
=[442+312\sqrt{2}, 
\overline{-298532+211094\sqrt{2},884+624\sqrt{2}}], 
\hat{P}_5=\sqrt{2+\sqrt{2}}$}}
{P\textfiveinferior=[442+312\textsurd{2},\textover -298532+211094\textsurd{2},\textover 884+624\textsurd{2}],
\^{P}\textfiveinferior =\textsurd(2+\textsurd{2})}}
\label{runt4}
The $\Z[\sqrt{2}]$-PCF $P_5$ from \cite[Cor.~8.19]{bej}
is of type $(N,k)=(1,2)$.  We have
\begin{align*}
E=E(P_5) & = D(442+312\sqrt{2})D(-298532+211094\sqrt{2})
D(884+624\sqrt{2})D(442+312\sqrt{2})^{-1}\\
&= \begin{bmatrix}
-228487+161564\sqrt{2} & -174876+123656\sqrt{2}\\
-298532+211094\sqrt{2} & -228487+161564\sqrt{2}\end{bmatrix},\text{ so}\\
\Quad(P_5)&=(-298532+211094\sqrt{2})x^2+174876-123656\sqrt{2}\\
&=(-298532+211094\sqrt{2})(x^2-2-\sqrt{2}),
\end{align*}
$\Roots(P_5)=\{ \beta\colonequals \sqrt{2 + \sqrt{2}}, -\beta\}$,
and the convergent
$\mathcal{C}_2(P_5)=(11502+8105\sqrt{2})/52\approx 441.6$,
which is a horrible approximation to $\beta$ in keeping with
\cite[p. 415]{bej}, which remarks on and quantifies the 
extremely slow convergence of  $P_5$.
The values of the characters 
\[
\mu\colonequals \lambda_\beta\circ\overline{E},
,\mu^\ast\colonequals \lambda_{-\beta}\circ\overline{E}
:\oPCF_{\Z[\sqrt{2}]}(P_5)\rightarrow \Z[\sqrt{2}][\beta]^\times
\]
on the class $\overline{P}_5\in\oPCF_{\Z[\sqrt{2}]}(P_5)$
are
\begin{align*}
\mu(\overline{P}_5)&=(-228487+161564\sqrt{2})+(-298532+211094\sqrt{2})\beta\text{ and}\\
\mu^\ast(\overline{P}_5)&=(-228487+161564\sqrt{2})-(-298532+211094\sqrt{2})
\beta,
\end{align*}
which are both units in the ring $\Z[\sqrt{2}][\beta]=\Z[\beta]$.
In fact
\[
\Norm_{\Q(\beta)/\Q(\sqrt{2})}\mu(\overline{P}_5)=\Norm_{\Q(\beta)/\Q(\sqrt{2})}\mu^\ast(\overline{P}_5)=
\mu(\overline{P}_5)\mu^\ast(\overline{P}_5)=
\det(\overline{E}(\overline{P}_5))=1.
\]

\section{Refined notions of convergence for periodic continued fractions}
\label{sec:convergePCF}
For a PCF $P$, the multiset $\Roots(P)$ contains the limit
$\hat P$ of $P$ when that limit exists; see, e.g.,
\cite[Appendix]{bej}.

\begin{theorem}
\label{ratify}
Let $P=[b_1,\ldots, b_N,\overline{a_1,\ldots, a_k}]$
and $P'=[b'_1,\ldots, b'_{N'},\overline{a'_1,\ldots , a'_{k'}}]$
be PCFs of types $(N,k)$ and $(N',k')$ that are $\CF$-equal.
\begin{enumerate}
\item\label{ratify1} If $k'=k$ so that $P$ and $P'$ are $\kk$-equal, then
$E(P)=E(P')$, $\Quad(P)=\Quad(P')$,
and $\Roots(P)=\Roots(P')$.
\item\label{ratify2} If $k'\ne k$, then $E(P)^{k'}=E(P')^k$, and
$\Quad(P)$ and $\Quad(P')$ are linearly dependent.
If $\Quad(P)$ and $\Quad(P')$ are both nonzero, then
$\Roots(P)=\Roots(P')$.
\item\label{ratify3}
The PCF $P$ converges if and only if $P'$ converges and in this case
their limits are equal: $\hat P=\hat P'\in\Roots(P)=\Roots(P')$.
\end{enumerate}
\end{theorem}

\begin{proof}
\eqref{ratify1}: Suppose $k=k'$. If $N=N'$, then $P$ equals $P'$.
In the alternative $N \neq N'$, we may assume $N'>N$ with $N'=N+j$.
In this case,
$$P'=[b_1,\ldots, b_{N},b_{N+1},\ldots,b_{N+j},\overline{a_{j+1},\ldots , a_{j+k}}],$$
where we take the subscripts of the $a$'s $\bmod\ k$ to put them in the range 1 to $k$
and $b_i=a_{i-N}$ for $i>N$.
So
\begin{align*}
E(P')&=M([b_1,\ldots, b_{N},b_{N+1},\ldots,b_{N+j},a_{j+1},\ldots , a_{j+k},0,
-b_{N+j},\ldots,-b_1,0])\\
&=M([b_1,\ldots, b_{N},a_1,a_2,\ldots, a_{j+k},0,
   -a_j,\ldots,-a_1,-b_N,\ldots,-b_1,0])\\
&=M([b_1,\ldots, b_{N},a_1,a_2,\ldots, a_{j+k-1},0,
   -a_{j-1},\ldots,-a_1,-b_N,\ldots,-b_1,0])\\
& \qquad\qquad\qquad\qquad\qquad\qquad \vdots\\
&=M([b_1,\ldots, b_{N},a_1,\ldots, a_k,0,-b_{N},\ldots,-b_1,0])=E(P)
\end{align*}
using the fact that $M([a,0,-a])=M([0])$ and $a_i=a_{i+k}$ $j$ times.
The rest of \eqref{ratify1} follows trivially.

\eqref{ratify2}:  Since the choice of $N$ does not affect $E$,
we may assume $N=N'$.
The equality $E(P)^{k'}=E(P')^k$ follows from the observation that the concatenation of
$k'$ copies of $[a_1,\ldots, a_k]$ equals the concatenation of $k$ copies of
$[a'_1,\ldots , a'_{k'}]$.
The Cayley--Hamilton theorem implies that any power of a $2\times2$
matrix $M$ is a linear combination of $M$ and $I$.
This implies $E(P)$, $E(P')$, and $\Idd$ are
linearly dependent, which implies $\Quad(P)$ and $\Quad(P')$ are linearly dependent,
which implies they have the same roots if both are nonzero.

\eqref{ratify3}: This is trivial,
since $P$ and $P'$ are $\CF$-equal.
\end{proof}

For the PCF $P=[b_1,\ldots, b_N,\overline{a_1,\ldots, a_k}]$
of type $(N,k)$, we introduce notation 
for the matrix $M_{N+j}(P)^{-1}M_{N+j+k}(P)$ and its entries:
$$
 R_j(P) = M_{N+j}(P)^{-1}M_{N+j+k}(P) = M([a_{j+1},\ldots, a_k, a_1,\ldots, a_j])=
\begin{bmatrix} r_j(P) & s_j(P) \\ t_j(P) & u_j(P) \end{bmatrix},
$$
and for the limits, if they exist,
\begin{equation}
\label{stuffing}
\hat{P}_j = \lim_{n\to\infty} \CC_{N+j+nk}(P).
\end{equation}

\begin{definition1}
The matrix  $R_j(P)$ is \textsf{heavy} when
\begin{equation*}
t_j(P) = 0\ \textrm{and}\ |u_j(P)| > 1.
\end{equation*}
Rephrased with this terminology, the condition
\cite[Thm.~4.3(b), INEQ 4.1]{bej} says that $R_j(P)$ is heavy for some
$0\le j\le k-1$.
\end{definition1}

The matrix $R_0(P)$ is $M(\Rep(P))$. The matrices $R_j(P)$ are all conjugate to $E(P)$.
From the definition, we see $\hat{P}_j = \hat{P}_{j+k}$.

\begin{proposition}
\label{betaMinus}
Let $P$ be a PCF.  Suppose $R_j(P)$ is heavy.
\begin{enumerate}
\item \label{betaMinus1} 
The limit $\hat{P}_j$ exists 
and $v(\hat{P}_j)$ is an $r_j(P)$-eigenvector of $E(P)$.
\item\label{betaMinus2} 
$R_{j+1}(P)$ is not heavy.
\end{enumerate}
\end{proposition}

\begin{proof}
\eqref{betaMinus1}:
The heavy condition implies $R_j(P)$ has an eigenvector
$\vect{1}{0}$ with eigenvector $r_j(P)$.  Therefore,
$E=M_{N+j}R_{j}M_{N+j}(P)^{-1}$ has eigenvector $M_{N+j}(P)\vect{1}{0}$.
We compute the limit
\[
\hat{P}_j=\lim_{n\to\infty}M_{N+j+nk}(P)\infty
=\lim_{n\to\infty}M_{N+j}(P)R_j^n\infty=M_{N+j}\infty.
\]
We are now done because $v(\hat{P}_j)$ and $M_{N+j}(P)\vect{1}{0}$ are both
elements of the equivalence class $M_{N+j}(P)\infty$.

\eqref{betaMinus2}: The calculation, with argument $P$ suppressed and $a = a_{j+1-nk}$, for the value of
$nk$ such that $1 \leq j+1-nk \leq k$,
$$
\begin{bmatrix}r_{j+1} & s_{j+1} \\ t_{j+1} & u_{j+1}\end{bmatrix} =
\begin{bmatrix}0 & 1 \\ 1 & -a\end{bmatrix}
\begin{bmatrix}r_j & s_j \\ 0 & u_j\end{bmatrix}
\begin{bmatrix}a & 1 \\ 1 & 0\end{bmatrix}
=
\begin{bmatrix}u_j  & 0 \\ 
(r_j-u_j)a+s_j & r_j\end{bmatrix}
$$
shows that $R_j(P)$ and $R_{j+1}(P)$ cannot both be heavy.
If they are, $u_{j+1} = r_{j}$, yielding $|u_{j+1}| < 1$, a contradiction.  
\end{proof}

We recall a theorem proved in the appendix of \cite{bej}, with the cases aligned to 
match Proposition \ref{tetrapartite}. 

\begin{theorem}\label{bactrian}
\textup{(\cite[Thm.~4.3]{bej}) }
Let $P$ be a PCF, let $\lambda_\pm$ be the eigenvalues of $E = E(P)$ chosen
so that $|\lambda_+| \ge 1\ge|\lambda_-|$, and let $v(\beta_\pm)$ be the
corresponding eigenvectors.
If $P$ converges, its limit $\hat P=\beta_+$.
Exactly one of the following holds.
\begin{enumerate}
\item\label{bactrian1}
$\Quad(P)$ has one root $\beta_+=\beta_-$ of multiplicity $2$, 
and
$\hat{P}=\beta_+=\beta_-.$

\item\label{bactrian2}
$\Quad(P) = 0$, $E=\lambda_+\Idd=\lambda_-\Idd$, and $P$ diverges.

\item\label{bactrian3}
$\Quad(P)$ has roots $\beta_+, \beta_-$ with $|\lambda_+| > |\lambda_-|$.

\begin{enumerate}
\item\label{bactrian3a}
For some $j \geq 0$, $R_j(P)$ is heavy, $\hat{P}_j = \beta_-$, 
$\hat{P}_{j+1}=\beta_+$,
and $P$ diverges.
\item\label{bactrian3b}
For all $j\geq 0$, $R_j(P)$ is not heavy, and $P$ converges to $\hat P = \beta_+$.
\end{enumerate}
\item\label{bactrian4}
$\Quad(P)$ has distinct roots $\beta_+, \beta_-$ with $|\lambda_+|=|\lambda_-|$,
and $P$ diverges. In fact, $\hat{P}_j$ does not exist for some $j$.
\end{enumerate}
\end{theorem}

\begin{proof}
In case \eqref{bactrian1}, by Proposition \ref{tetrapartite}\eqref{tetrapartite1},
for all $j$, $\hat{P}_j= \beta_+$, so $P$ converges to $\beta_+$.

In case \eqref{bactrian2},
$R_{j}(P) = E$ for all $j,$ so
$\hat{P}_j= \CC_{N+j}(P)$.
No two consecutive ones are equal: $\hat{P}_j\ne\hat{P}_{j+1}$ by the 
more general
rule that no two consecutive convergents can be equal, so $P$ diverges.

Now assume that cases \eqref{bactrian1} and \eqref{bactrian2} do not hold.
$\Quad(P)$ has two distinct roots, $\beta_-$, $\beta_+$.  The matrix
$E$ has corresponding distinct eigenvalues
$\lambda_-$, $\lambda_+$.  

In case \eqref{bactrian3}, $|\lambda_+| > |\lambda_-|,$ and
Proposition 
\ref{tetrapartite}\eqref{tetrapartite3} shows all $\beta_j$ exist.
Suppose, as in the first subcase, there exists $j$ such that $R_j(P)$ is heavy,
Proposition \ref{betaMinus} shows $\hat{P}_j= \beta_-$ and $R_{j+1}$ is not heavy.
Proposition \ref{tetrapartite}\eqref{tetrapartite3} 
shows $\hat{P}_{j+1}=\beta_+,$ so $P$ diverges.

In the alternative subcase of \eqref{bactrian3}, where there
is no $j$ such that $R_j$ is heavy, we have $\hat{P}_j= \beta_+$ for all $j$, 
and
$P$ converges to $\beta_+$.

In case (\ref{bactrian4}) of divergence, we have  $|\lambda_+|=|\lambda_-|$
and $\lambda_+\ne\lambda_-$.  By Proposition \ref{tetrapartite}\eqref{tetrapartite4},
the limit $\hat{P}_j$ exists if and only if $t_j=t_j(P)=0$.
We claim there exists $j$ such that $t_j(P)\ne0$.

Assume, to the contrary, that $t_j = 0$ for all $j$.
To deduce a contradiction, we show $t_j=t_{j+1}=t_{j+2}=0$ implies $a_{j+2}=0$.
From the equality
\begin{equation}\label{pie}
R_{j+1}\begin{bmatrix}a_{j+2} & 1 \\ 1 & 0\end{bmatrix}
=\begin{bmatrix}a_{j+2} & 1 \\ 1 & 0\end{bmatrix}R_{j+2},
\end{equation}
we see that $t_{j+1}=t_{j+2}=0$ implies $s_{j+2}=0$.
Similarly, $t_{j}=t_{j+1}=0$ implies $s_{j+1}=0$, so $R_{j+1}$ and $R_{j+2}$ are
diagonal with distinct diagonal entries $\lambda_\pm$.
Equation \eqref{pie} implies 
$a_{j+2}=0$.

Since $t_j = 0$ for all $j$, we have that $a_j = 0$ for all $j$.
If $k$ is odd then $t_j=1$ for every j, a contradiction.
If $k$ is even then we are in case (\ref{bactrian2}), also a contradiction,
thus proving the claim.
Hence, for some $j$, $\hat{P}_j$ does not exist, and $P$ diverges.
\end{proof}

We name the convergence behaviors of a PCF $P$ identified in Theorem \ref{bactrian}.

\begin{definition1} 
\label{delivery}
In cases \ref{bactrian}\eqref{bactrian1} and
\ref{bactrian}\eqref{bactrian3}, $P$ is \textsf{quasiconvergent}.
In cases \ref{bactrian}\eqref{bactrian2} and
\ref{bactrian}\eqref{bactrian4}, $P$ is \textsf{strictly divergent}.
In the divergent subcase of \ref{bactrian}\eqref{bactrian3a},
$P$ is \textsf{strictly quasiconvergent}.
\end{definition1}
\noindent Divergent means strictly divergent or strictly quasiconvergent.
Quasiconvergent means convergent or strictly quasiconvergent.

\begin{theorem}
\label{tree}
Let PCFs $P$ and $Q$ be equivalent.
$P$ is quasiconvergent if and only if $Q$ is quasiconvergent.
If $P$ and $Q$ both converge, then 
$\hat P =\hat Q$.
\end{theorem}

\begin{proof}
PCFs $P$ and $Q$ are quasiconvergent together,
because $E(P) = E(Q)$.
For the same reason, if $P$ and $Q$ both converge, they converge to the same value.
\end{proof}
The following lemma and theorem characterize strict quasiconvergence.

\begin{lemma}
\label{stepbytwo}
Let $P=[b_1,\ldots, b_N, \overline{a_1,\ldots, a_k}]$ be a PCF.  If
$R_j(P)$ and $R_{j+2}(P)$ are both heavy, then $a_{j+2-n'k}=0$ for the value of $n'$ such that
$1 \leq j+2-n'k \leq k$.
\end{lemma}

\begin{proof}
Inspect the relation, with argument $P$ suppressed, $a = a_{j+1-nk}$, and $a' = a_{j+2-n'k}$:
\[
\begin{bmatrix} r_{j+2} & s_{j+2} \\ t_{j+2} & u_{j+2} \end{bmatrix} = 
\begin{bmatrix} 1 & -a \\ -a' & aa'+1 \end{bmatrix}
\begin{bmatrix} r_{j} & s_{j} \\ t_{j} & u_{j} \end{bmatrix} 
\begin{bmatrix} aa'+1 & a \\ a' & 1 \end{bmatrix}.
\]
Observe that $t_j=0$ implies
\begin{gather*}
r_{j+2} = r_j + aa'r_j -aa'u_j + a's_j,\\
t_{j+2} = a'(u_j-r_{j+2}).
\end{gather*}
If $t_{j+2} = 0$ and $a' \ne 0$, then $r_{j+2} = u_{j}$, contradicting
$|u_j|>1>|r_{j+2}|$.
Hence, $t_{j+2}=0$ implies $a' = a_{j+2-n'k} = 0.$
\end{proof}

\begin{theorem}
\label{charquasicon}
The PCF $P$ is strictly quasiconvergent if and only if
for each $j$, $1\leq j \leq k$,  $\hat{P}_j$
exists, a strict majority, but not all, of which are the same.
\end{theorem}

\begin{proof}
Let the eigenvalues of $E = E(P)$ be $\lambda_+$ and $\lambda_-$, with $|\lambda_+| \geq 1$.
If both eigenvalues have magnitude $1$, arbitrarily assign one to be $\lambda_+$.
Define $\beta_+$ and $\beta_-$ such that $v(\beta_+)$, $v(\beta_-)$
are eigenvectors of $E$  with corresponding eigenvalues $\lambda_+$ and $\lambda_-$.

If $P$ is strictly quasiconvergent, then there exists $j$ such
that $R_j(P)$ is heavy and
 $\lambda_+ = u_j.$

Proposition \ref{betaMinus} shows that
for every $j$ such that $R_j(P)$ is heavy,
$\hat{P}_j = \beta_-$
and $R_{j+1}(P)$ is not heavy.
For those $j$ such that $R_j(P)$ is not heavy,
$
\hat{P}_j = \beta_+.
$

If $R_j(P)$ is heavy for every even $j$ or every odd $j$,
then either $k$ is even and, by Lemma \ref{stepbytwo}, every second
partial quotient of $\Rep(P)$ is zero, or $k$ is odd, and all partial quotients of $\Rep(P)$
are zero.
In the even case, $E$ is conjugate to an elementary matrix and $P$ satisfies 
\ref{bactrian}\eqref{bactrian1} or \ref{bactrian}\eqref{bactrian2}, a contradiction.
In the odd case, $E$ is conjugate to the matrix $J$ 
of Remark \ref{butter} and $P$ satisfies \ref{bactrian}\eqref{bactrian4}.
Hence, for only a minority of $j$, $1 \leq j \leq k$,
can $R_j(P)$ be heavy and $\hat{P}_j= \beta_-$.

Conversely, suppose
for each $j$, $1\leq j \leq k$,  $\hat{P}_j$
exists, a strict majority, but not all, of which are the same.

If $P$ satisfies Theorem \ref{bactrian}(\ref{bactrian2}), then
$\hat{P}_j \ne \hat{P}_{j+1}$.
So a majority of the $\hat{P}_j$, $1\leq j\leq k$, are not equal to a
particular value.

The PCF $P$
satisfies Theorem \ref{bactrian}(\ref{bactrian4}) if and only if
$E$ has eigenvalues $\lambda_+\ne\lambda_-$ with $|\lambda_+|=|\lambda_-|=1$.
We claim that for $1\leq j\leq k$, either $\hat{P}_j$ does not exist, or 
$\hat{P}_j$ exists and
 $\CC_{N+j+nk}(P)$ is a constant independent of $n$.
We shall prove this for $j=k$.  The result follows for all $j$ by rotating the
periodic part of $P$.
Write
$$v(\hat{P}_k) = c v(\beta_+) + d v(\beta_-),$$
and observe
$$E^n v(\hat{P}_k) = c \lambda_{+}^n v(\beta_+) + d \lambda_{-}^n v(\beta_-).$$
If $\hat{P}_k$ exists, then $v(\hat{P}_k)$ is an eigenvector of $E$, 
so either $c=0$ or $d=0$.
This proves the claim.

By assumption that the limits $\hat{P}_j$ exist and the fact that consecutive
convergents cannot be equal, we cannot have a strict majority of $\hat{P}_j$ agree and
$P$ satisfy Theorem \ref{bactrian}(\ref{bactrian4})

The only possibility left is that $P$ satisfies
Theorem \ref{bactrian}(\ref{bactrian3a}),
so $P$ is strictly quasiconvergent.
\end{proof}

\section{Examples for quasiconvergence}
\label{examples}

\subsection{Strict majority in Theorem \ref{charquasicon} is necessary}
\label{pumpkin1}
The strict majority hypothesis in the sufficiency part of the proof of
Theorem \ref{charquasicon}  is essential.
For example, $P = [a,b,\overline{0,0}]$ has limits $\hat{P}_1=a$
and $\hat{P}_2=a+1/b$.  The split is equal and
$P$ is not quasiconvergent.
It is strictly divergent because it satisfies Theorem \ref{bactrian}\eqref{bactrian2}.

\subsection{The fraction of \texorpdfstring{\except{toc}{\boldmath{$\hat{P}_j=\beta_+$}}\for{toc}{$\hat{P}_j=\beta_+$}}
{\textbeta\textjinferior(P)=\textbeta\textplusinferior} can be made arbitrarily close to \texorpdfstring{\except{toc}{\boldmath{$1/2$}}\for{toc}{$1/2$}}{\textonehalf}}
\label{pumpkin2}
We can make the fraction of $\hat{P}_j= \beta_+$ arbitrarily close to $1/2$.
Let $c \ne 0$ and
$$P=[a,b,\overline{c,-1/c,c,0,\ldots,0}\,],$$
have odd length period $k\geq 3$.
We have $\hat{P}_{2j}=a$ and $\hat{P}_{2j+1}=a+1/b$, $1\le j < k/2$.
If $c^2=-1$ we are in case \eqref{bactrian2} of Theorem \ref{bactrian},
and $\hat{P}_1=a+1/(b-c)$.
If $|c|=1$, $c^2 \ne -1$, we are in case \eqref{bactrian4}, and $\hat{P}_1$
does not exist.
If $|c|\ne1$ we are in case \eqref{bactrian3a} and
\[\hat{P}_1=\beta_+=\begin{cases}a+1/b&\text{if }|c|>1,\\ a&\text{if }|c|<1,\end{cases}\]
so $\beta_+$ has a slim $(k+1)/2$ majority.
A similar construction can be made for even periods by taking
$$P'=[a,b,\overline{c,-1/c,c,0,\ldots,0,d,-1/d,d,0,\ldots,0}\,]$$
where the length of the strings of 0's have the same parity.

\subsection{The fraction of \texorpdfstring{\except{toc}{\boldmath{$\hat{P}_j=\beta_+$}}\for{toc}{$\hat{P}_j=\beta_+$}}
{\textbeta\textjinferior(P)=\textbeta\textplusinferior} can be made arbitrarily close to \texorpdfstring{\except{toc}{\boldmath{$1$}}\for{toc}{$1$}}{1}}
\label{pumpkin3}
Likewise we can make the fraction of $\hat{P}_j=\beta_+$ arbitrarily
close to $1$.  Let $\{F_n\}_{n=1}^\infty =1,1,2,3,5,\ldots $ be the
\textsf{Fibonacci sequence}. Let $b\neq 0$ and 
\[
P\colonequals Q_k=[a,b,\overline{1, \ldots, 1, -F_{k-2}/F_{k-1}}\,],
\]
where there are $(k-1)$ $1$'s.
The continued fraction $P$ for $k\ge4$ provides an example where exactly
$k-1$ of the $k$ $\hat{P}_j$'s are $\beta_+$.
(If $k=3$, $\hat{P}_2$ doesn't exist, and if $k=2$, $\beta_+=\beta_-=\hat{P}$.)

\begin{proposition}
If $k\geq 4$ and $b\ne0$, then
$\hat{P}_j=\beta_+(Q_k)=a+\frac{F_{k-1}}{bF_{k-1}+F_{k-2}}$
for $1\le j\le k-1$, but $\hat{P}_k=\beta_-(Q_k)=a+1/b$.
\end{proposition}

\begin{proof}
 We compute
$M=M(k)\colonequals M([1,1,\dots, 1, -F_{k-2}/F_{k-1}])$:
\begin{align}
\label{turkey}
M &=\begin{bmatrix}1&1\\1&0\end{bmatrix}^{k-1}
\begin{bmatrix} -F_{k-2}/F_{k-1} &1\\1 &0\end{bmatrix}\nonumber\\
&=
\begin{bmatrix} F_k & F_{k-1}\\F_{k-1} & F_{k-2}\end{bmatrix}
\begin{bmatrix} -F_{k-2}/F_{k-1} &1\\1 &0\end{bmatrix}\nonumber\\
&= 
\begin{bmatrix} (-1)^k/F_{k-1} & F_k\\0 & F_{k-1}\end{bmatrix},
\end{align}
using Cassini's identity $F_{k-1}^2-F_{k}F_{k-2}=(-1)^k$ (\cite[p.~81]{Knuth}).

The next step is to compute
 the $\hat{P}_j$, $1\leq j\leq k$, as in \eqref{stuffing}.
We begin with $\hat{P}_k$.
Set 
\begin{equation}
\label{broke}
A=D(a)D(b)=
\begin{bmatrix}a & 1\\1 & 0\end{bmatrix}
\begin{bmatrix}b & 1\\1 & 0\end{bmatrix}=
\begin{bmatrix}
ab + 1 & a \\b & 1\end{bmatrix}.
\end{equation}
Using \eqref{turkey}, we have
\begin{align*}
M(n,k)\colonequals M_{2+nk}(Q_k)&= 
AM^n = \begin{bmatrix}
ab + 1 & a \\b & 1\end{bmatrix}
\begin{bmatrix}
(-1)^k/F_{k-1} & F_k\\ 0 & F_{k-1}\end{bmatrix}^n\\
&= \begin{bmatrix}ab+1 & a\\b & 1\end{bmatrix}\begin{bmatrix}
(-1)^{kn}/F_{k-1}^n & \ast\\0& F_{k-1}^n\end{bmatrix}\\
&=\begin{bmatrix} (ab+1)(-1)^{kn}/F_{k-1}^n & \ast\\b (-1)^{kn}/F_{k-1}^n & \ast
\end{bmatrix}.
\end{align*}
Hence 
\begin{equation*}
\mathcal{C}_{N+nk}=M(n,k)_{11}/M(n,k)_{21}=\frac{ab+1}{b}=a + 1/b
\end{equation*}
for all $n\geq 1$ and
\begin{equation*}
\hat{P}_k=\lim_{n\to\infty}\mathcal{C}_{N+nk}=a + 1/b.
\end{equation*}

For the other $\hat{P}_j$, we need to calculate
$M_j\colonequals D(1)^{-j}M D(1)^j$ for $1\leq j \leq k-1$,
so that we can apply Proposition \ref{tetrapartite}.
Firstly note that
\begin{equation}
\label{gravy}
D(1)^j = \begin{bmatrix} F_{j+1} & F_j\\F_j & F_{j-1}\end{bmatrix},\quad
\text{and so} \quad D(1)^{-j}=(-1)^j\begin{bmatrix} F_{j-1}& -F_j\\
-F_j & F_{j+1}\end{bmatrix}.
\end{equation}
Hence from \eqref{gravy} and \eqref{turkey} we get
\begin{align}
\label{brussel}
M_j &=(-1)^j\begin{bmatrix} F_{j-1}& -F_j\\
-F_j & F_{j+1}\end{bmatrix}
\begin{bmatrix} (-1)^k/F_{k-1} & F_k\\0 & F_{k-1}\end{bmatrix}
\begin{bmatrix} F_{j+1} & F_j\\F_j & F_{j-1}\end{bmatrix}\nonumber\\
&= (-1)^j\begin{bmatrix} \ast & \ast \\
F_j((-1)^{k+1}F_{j+1}/F_{k-1}+F_{j+1}F_{k-1}-F_{j}F_{k}) & \ast\end{bmatrix}.
\end{align}
If $(M_j)_{21}=0$, then we would have $F_{k-1}|F_{j+1}$ from \eqref{brussel}.  
But this is 
not possible for $j\leq k-1$ unless $j=k-2$: 
$F_{k-1}\not | F_{k}$ takes care of $j=k-1\geq 3$ and if $j\leq k-3$,
then $F_{k-1}>F_{j+1}$ since $k\geq 4$.
Hence $(M_j)_{21}\neq 0$ if $1\leq j\leq k-1$, $j\neq k-2$,
so in this case $\hat{P}_j=\beta_+(Q_k)$
by Proposition \ref{tetrapartite}\eqref{tetrapartite3},
with the ``$M$'' there equal to $AMA^{-1}$ and the ``$\beta$'' there
equal to $AD(1)^j\infty$.

For the exceptional case $j=k-2$, explicit computation of $M_{k-2}$ using
the Cassini/Vajda identities for Fibonacci numbers gives
\begin{equation}
\label{cranberry}
M_{k-2}=\begin{bmatrix} F_{k-1} & F_kF_{k-3}/F_{k-1}\\
0 & (-1)^k/F_{k-1}\end{bmatrix}.
\end{equation}
Hence by \eqref{cranberry} $(M_{k-2})_{21}=0$ and 
$|(M_{k-2})_{11}|>|(M_{k-2})_{22}|$ since $k\geq 4$.
Hence $\hat{P}_{k-2}=\beta_+(Q_k)$ again by
Proposition \ref{tetrapartite}\eqref{tetrapartite3},
with the ``$M$'' there equal to $AMA^{-1}$ and the ``$\beta$'' there
equal to $AD(1)^{k-2}\infty$.

Lastly we show the computation giving $\beta_{-}(Q_k)$ and $\beta_+(Q_k)$:
\begin{equation}
\label{ratify4}
\beta_-(Q_k)=a+ 1/b\quad\text{and}\quad \beta_+(Q_k)=
a+\frac{F_{k-1}}{bF_{k-1}+F_{k-2}}.
\end{equation}
Calculating  $\beta_{-}(Q_k)$ and $\beta_{+}(Q_k)$ entails finding
the eigenvectors of  $E(Q_k)=AMA^{-1}$ with $A$ as in \eqref{broke} and $M$
as in \eqref{turkey}.
Let 
\begin{equation*}
v_{-}\colonequals\begin{pmatrix}1\\0\end{pmatrix},\,\, \lambda_{-}\colonequals
\frac{(-1)^k}{F_{k-1}},
\quad\text{and}\quad
v_{+}\colonequals\begin{pmatrix} F_{k-1}\\F_{k-2}\end{pmatrix},\,\,
\lambda_{+}\colonequals F_{k-1}.
\end{equation*}
Observe that $|\lambda_+|>|\lambda_{-}|$ since $k\geq 4$.
Using the Cassini identity yet again, verify that
\begin{equation*}
Mv_{-}=\lambda_{-}v_{-}\quad\text{and}\quad Mv_{+}=\lambda_{+}v_{+}.
\end{equation*}
So $E(Q_k)=AMA^{-1}$ will have eigenvectors $Av_{-}$, $Av_{+}$
with eigenvalues $\lambda_{-}$, $\lambda_{+}$, respectively.
We have 
\begin{align*}
Av_{-}&=\begin{pmatrix}ab+1\\b\end{pmatrix}=b\begin{pmatrix}
a+1/b \\1\end{pmatrix}=bv(\beta_{-})\text{ and}\\
Av_{+}&=\begin{pmatrix}(ab+1)F_{k-1}+aF_{k-2}\\bF_{k-1}+F_{k-2}\end{pmatrix}=
(bF_{k-1}+F_{k-2})\begin{pmatrix}a+\frac{F_{k-1}}{bF_{k-1}+F_{k-2}}\\1\end{pmatrix}\\
&=(bF_{k-1}+F_{k-2})v(\beta_{+}),
\end{align*}
establishing the formulas \eqref{ratify4} for $\beta_{-}(Q_k)$ and
$\beta_{+}(Q_{k})$.
\end{proof}

\subsection{The equivalence 
\texorpdfstring{\except{toc}{\boldmath{$\sim$}}\for{toc}{$\sim$}}{\textasciitilde}
on 
\texorpdfstring{\except{toc}{\boldmath{$\PCF(\OO)$}}\for{toc}{$\PCF(\OO)$}}{PCF(O)} 
does not respect convergence}
\label{pumpkin4}
Theorem \ref{tree} shows quasiconvergence is a property of 
PCF equivalence class.
It is possible for $P\sim Q$, both quasiconvergent,
with  $P$ strictly quasiconvergent (and hence divergent)
and $Q$ convergent.
For example, take $a\ne 0$ and let $P=[\overline{a,0,-1/a}]$, $Q =[\overline{a-1/a}]$.
Let the convergents of $P$ be
$\mathcal{C}_i(P)$, $i\geq 1$.
If $|a|<1$, then $P$ satisfies Theorem \ref{bactrian}\eqref{bactrian3a}, 
and hence is strictly quasiconvergent.
The limits $\hat{P}_j$ for $1\leq j\leq 3$ are 
\begin{equation*}
\begin{array}{ccccc}
\hat{P}_1&=&\lim_{i\to\infty}\mathcal{C}_{1+3i}(P)&=&a, \\
\hat{P}_2&=&\lim_{i\to\infty}\mathcal{C}_{2+3i}(P)&=&-1/a, \\
\hat{P}_3&=&\lim_{i\to\infty}\mathcal{C}_{3+3i}(P)&=&-1/a.
\end{array}
\end{equation*}
The PCF $Q$ converges to $-1/a$, since the
convergents are just $\mathcal{C}_i(Q)=\mathcal{C}_{3i}(P)$.

\bibliographystyle{alpha}
\bibliography{compo}
\end{document}